\documentclass[12pt]{amsart}
\usepackage{amsfonts, amssymb, amscd, amsthm, latexsym, array}
\usepackage{fullpage,color}
\usepackage{verbatim,euscript}
\usepackage[matrix,arrow,curve]{xy}

\numberwithin{equation}{section}

\usepackage[colorlinks=true,linkcolor=blue,urlcolor=violet,citecolor=magenta]{hyperref}

\input {cyracc.def}
\font\tencyr=wncyr10 
\font\tencyi=wncyi10 
\font\tencysc=wncysc10 
\def\rus{\tencyr\cyracc}
\def\rusi{\tencyi\cyracc}
\def\rusc{\tencysc\cyracc}

\newtheorem{thm}{Theorem}[section]

\newtheorem{lm}[thm]{Lemma}
\newtheorem{cl}[thm]{Corollary}
\newtheorem{prop}[thm]{Proposition}

\theoremstyle{remark}
\newtheorem{rmk}[thm]{Remark}

\theoremstyle{definition}
\newtheorem{ex}[thm]{Example}
\newtheorem{df}{Definition}
\newtheorem*{rema}{Remark}

\newcommand {\g}{{\mathfrak g}}

\newcommand {\fN}{\mathfrak{N}}

\newcommand {\ut}{{\mathfrak u}}

\newcommand {\gc}{{\mathcal C}}

\newcommand {\co}{{\mathcal O}}

\newcommand {\BA}{{\mathbb A}}
\newcommand {\BP}{{\mathbb P}}
\newcommand {\BZ}{{\mathbb Z}}
\newcommand {\BN}{{\mathbb N}}
\newcommand {\BQ}{{\mathbb Q}}
\newcommand {\BR}{{\mathbb R}}

\newcommand {\sfr}{{\mathsf R}}

\newcommand {\md}{/\!\!/}
\newcommand {\isom}{\stackrel{\sim}{\longrightarrow}}

\newcommand{\lb}{\lambda}
\newcommand{\ap}{\alpha}

\renewcommand{\le}{\leqslant}
\renewcommand{\ge}{\geqslant}

\newcommand{\curle}{\preccurlyeq}

\newcommand{\eus}{\EuScript}

\newcommand {\Ann}{{\mathsf{ann}}}

\newcommand {\codim}{{\mathrm{codim\,}}}

\newcommand {\ed}{\mathrm{edim\,}}
\newcommand {\gr}{{\mathsf{gr}}}
\newcommand {\hd}{\mathsf{hd}}
\newcommand {\hot}{{\mathrm{ht\,}}}

\newcommand {\Lie}{{\mathrm{Lie\,}}}

\newcommand {\rk}{{\mathsf{rk\,}}}

\newcommand {\spe}{{\mathsf{Spec\,}}}
\newcommand {\srk}{{\mathsf{srk\,}}}

\newcommand {\GR}[2]{{\textrm{{\bf #1}}}_{#2}}

\newcommand {\ov}{\overline}
\newcommand {\un}{\underline}
\newcommand {\beq}{\begin{equation}}
\newcommand {\eeq}{\end{equation}}
\newcommand {\bfp}{{\boldsymbol{p}}}
\newcommand {\bv}{{\boldsymbol{v}}}
\newcommand {\bof}{{\boldsymbol{f}}}
\newcommand {\bbk}{\Bbbk}

\begin{document}
\setlength{\parskip}{2pt plus 4pt minus 0pt}
\hfill {\scriptsize September 14, 2011}
\vskip1ex

\title[Quotients by $U'$-actions]%
{Quotients by actions of the derived group of a maximal unipotent subgroup}
\author[D.\,Panyushev]{Dmitri I.~Panyushev}
\address[]{%
Institute for Information Transmission Problems of the R.A.S., \hfil\break\indent
 B. Karetnyi per. 19, Moscow 
127994, Russia}
\email{panyushev@iitp.ru}
\begin{abstract}
Let $U$ be a maximal unipotent subgroup of a connected semisimple group $G$ and $U'$ the derived group of $U$. If $X$ is an affine $G$-variety, then the algebra of $U'$-invariants, $k[X]^{U'}$, is finitely generated and the quotient morphism $\pi: X \to X\md U'=\spe
k[X]^{U'}$ is well-defined. In this article, we study properties of such quotient morphisms, e.g. the property that all the fibres of $\pi$ are equidimensional. We also establish an analogue of the Hilbert-Mumford criterion for the null-cones with respect to $U'$-invariants.
\end{abstract}
\maketitle

\section*{Introduction}
\noindent 
The ground  field $\bbk$ is algebraically closed and of characteristic zero.
Let $G$ be a semisimple  algebraic group with Lie algebra $\g$. 
Fix a maximal unipotent subgroup  $U\subset G$ and a maximal torus $T$ of the Borel
subgroup $B=N_G(U)$.  Set $U'=(U,U)$.
Let $X$ be an irreducible affine variety acted upon by $G$.
The algebra of covariants (or, $U$-invariants) $\bbk[X]^U$ is a classical and 
important object in Invariant Theory. It is known that $\bbk[X]^U$ is finitely generated 
and has many other useful properties and applications, see e.g. 
\cite[Ch.\,3, \S\,3]{kr84}. For a factorial conical variety $X$ with rational singularities, 
there are interesting relations between the Poincar\'e series of the graded algebras 
$\bbk[X]$  and  $\bbk[X]^U$, see \cite{br83}, \cite[Ch.\,5]{disser}.
Similar results for $U'$-invariants are obtained in \cite{odno-sv}.

A  surprising observation that stems from  \cite{odno-sv} is that, to a great extent, the theory
of $U'$-invariants is parallel to that of $U$-invariants. In this article, we elaborate on further
aspects of this parallelism. 
Our main object  is the quotient 
$\pi_{X,U'}:  X\to X\md U'=\spe\!(\bbk[X]^{U'})$. 
Specifically, we are interested in the property that
$X\md U'$ is an affine space and/or the morphism $\pi_{X,U'}$ is equidimensional (i.e.,
all the fibres of $\pi_{X,U'}$ have the same dimension). 
Our ultimate goal is to prove for $U'$ an analogue of the Hilbert--Mumford criterion
and to provide a classification of the irreducible representations $V$ of simple
algebraic groups $G$ such that $\bbk[V]$ is a free $\bbk[V]^{U'}$-module.
We also develop some theory for $U'$-actions on the affine
prehomogeneous horospherical varieties of $G$ ($\eus S$-{\it varieties\/} in terminology of 
\cite{vp72}). As $U'=\{1\}$ for $G=SL_2$, one sometimes has to assume that $G$ has no simple factors $SL_2$.

If $X$ has a $G$-fixed point, say $x_0$,
then the fibre of $\pi_{X,U'}$ containing $x_0$ is called
the {\it null-cone}, and we denote it by $\fN_{U'}(X)$.
(The null-cone $\fN_{H}(X)$ can be defined for any subgroup $H\subset G$
such that $\bbk[X]^H$ is finitely generated.)
If $G$ has no simple factors $SL_2$ nor $SL_3$, then
the canonical affine model of $\bbk[G/U']$ constructed in \cite[Sect.\,2]{odno-sv} consists of unstable points in the sense of GIT, and using this property we 
give a characterisation of $\fN_{U'}(X)$ in terms of one-parameter subgroups of $T$.
We call it the {\it Hilbert--Mumford criterion for $U'$.} This is inspired by similar
results of Brion for $U$-invariants \cite[Sect.\,IV]{br83}. It is easily seen that 
$\fN_{U'}(X)\subset \fN_{G}(X)$. Therefore $G{\cdot}\fN_{U'}(X)\subset \fN_{G}(X)$.
Using the Hilbert--Mumford criterion for $U'$ we prove that  
$G{\cdot}\fN_{U'}(X)= \fN_{G}(X)$ whenever
$G$ has no simple factors $SL_n$. This should be compared with the result of Brion 
\cite{br83} that
$G{\cdot}\fN_{U}(X)= \fN_{G}(X)$ for all $G$.

The $\eus S$-varieties are in  one-to-one correspondence with the finitely
generated monoids $\mathfrak S$ in the monoid $\mathfrak X_+$ of dominant weights,
and  the $\eus S$-variety corresponding to $\mathfrak S\subset \mathfrak X_+$ is denoted 
by $\gc(\mathfrak S)$.  We give exhaustive answers to three natural problems related 
to the  actions of $U'$ on $\eus S$-varieties. 
A set of fundamental weights $M$ is said to be {\it sparse\/} if the corresponding nodes 
of the Dynkin diagram are disjoint and, moreover, there does not exist any node (not in $M$)
that is adjacent to two nodes from $M$.
Our results are:

\begin{itemize}
\item[\sf a)] \  $\bbk[\gc(\mathfrak S)]^{U'}$ is a polynomial algebra {\sl if and only if\/}
the monoid $\mathfrak S$ is
generated by a set of fundamental weights;

\item[\sf b)] \    $\bbk[\gc(\mathfrak S)]^{U'}$ is a polynomial algebra and 
$\pi_{\gc(\mathfrak S), U'}$ is  equidimensional {\sl if and only if\/} the monoid $\mathfrak S$ is
generated by a {sparse} set of fundamental weights;

\item[\sf c)] \  the morphism $\pi_{\gc(\mathfrak S), U'}$ is equidimensional 
{\sl if and only if\/}  the convex polyhedral  cone $\BR^+\mathfrak S$ is generated by a 
sparse set of fundamental weights. 
(In particular, the cone $\BR^+\mathfrak S$ is simplicial.)
\end{itemize}
\noindent
Part a) is rather easy, while parts b) and c) require technical details related to
the Bruhat decomposition of the  flag variety associated with 
$\gc(\mathfrak S)$. 
If $\mathfrak S$ has one generator, say $\lb$, and $\sfr(\lb)$ is a simple $G$-module with
highest weight $\lb$,
then $\gc(\mathfrak S)$ is the closure of
the orbit of highest weight vectors in the dual $G$-module $\sfr(\lb)^*$. Such a variety is denoted by $\gc(\lb)$.
As in \cite{vp72}, we say that $\gc(\lb)$ is an \textsf{HV}-{\it variety}. Our results for
\textsf{HV}-varieties are more complete. For instance, we compute the homological
dimension of $\gc(\lb)\md U'$ and prove that $\fN_{U'}(\gc(\lb))$ is always of codimension~2
in $\gc(\lb)$. The criterion of part b) is then 
transformed into a sufficient condition applicable to a wider class of affine varieties:

\begin{thm}    \label{thm-intro:deformation}
Suppose that $G$ acts on  an irreducible affine variety $X$ such that 
(1) \ $\bbk[X]^{U}$ is a polynomial
algebra and  (2) \ the weights of free generators are fundamental, pairwise distinct, 
and form a sparse set.
Then $\bbk[X]^{U'}$ is also polynomial, of Krull dimension $2\dim X\md U$,
and the quotient $\pi_{X,U'}: X\to X\md U'$ is equidimensional.  
\end{thm}

This exploits the theory of ``contractions of actions'' of $G$ \cite{po86} and
can be regarded as a continuation of our work in~\cite[Sect.\,5]{aura}, where the 
equidimensionality  problem was considered for quotient morphism by $U$. 
For instance, under the hypotheses
of Theorem~\ref{thm-intro:deformation}, the morphism $\pi_{X,U}$ is also equidimensional.

In \cite{odno-sv}, we obtained a classification of the irreducible representations of simple 
algebraic
groups such that  $\bbk[V]^{U'}$ is a polynomial algebra. Now, using 
Theorem~\ref{thm-intro:deformation} and some {\sf ad hoc} arguments, 
we extract from that list the representations having the additional property that $\pi_{V,U'}$ is equidimensional.
The resulting list is precisely the list of representations such that 
$\bbk[V]$ is a free $\bbk[V]^{U'}$-module (such $G$-representations are said to be 
$U'$-{\it cofree}).

This work is organized as follows. Section~\ref{sect:recol} contains auxiliary results on 
$\eus S$-varieties \cite{vp72}, $U'$-invariants \cite{odno-sv}, and equidimensional morphisms. 
In Section~\ref{sect:HV}, we consider $U'$-actions on the \textsf{HV}-varieties.
Section~\ref{sect:S} is devoted to the $U'$-actions on arbitrary $\eus S$-varieties. Here we
prove results of items a) and b) above
(Theorems~\ref{thm:S-var}, \ref{thm:eq-sparse1}, and \ref{thm:eq-sparse2}). 
In Section~\ref{sect:Eq}, we prove the
general equidimensionality criterion for $\eus S$-varieties (item c)). The Hilbert--Mumford criterion for $U'$ and relations between two null-cones are discussed  
in Section~\ref{sect:HM}. In Section~\ref{sect:classif}, we prove Theorem~\ref{thm-intro:deformation} and obtain the classification of $U'$-cofree representations of $G$.
\vskip1.5ex

\un{\sl Notation}.
If an algebraic group $Q$ acts regularly on an irreducible affine variety $X$, then 
$X$ is called a $Q$-variety and

\textbullet \quad $Q_x=\{q\in Q\mid q{\cdot}x=x\}$ is the {\it stabiliser\/} of $x\in X$;

\textbullet \quad $\bbk[X]^Q$ is the algebra of $Q$-invariant polynomial functions on $X$.
If $\bbk[X]^Q$ is finitely generated, then $X\md Q:=\spe(\bbk[X]^Q)$, and
the {\it quotient morphism\/} $\pi_Q=\pi_{X,Q}: X\to X\md Q$ is the mapping associated with
the embedding $\bbk[X]^Q \hookrightarrow \bbk[X]$.

\noindent
Throughout, $G$ is a semisimple simply-connected algebraic group,  
$W=N_G(T)/T$ is the Weyl group, $B=TU$, and $r=\rk G$. Then 
\\ \indent
{\bf --}  $\Delta$ is the root system of $(G,T)$, 
$\Pi=\{\ap_1,\dots,\ap_r\}\subset \Delta$ are the simple roots corresponding to $U$, and 
$\varpi_1,\dots,\varpi_r$ are the corresponding fundamental weights. 
\\ \indent
{\bf --}  The character group of $T$ is denoted by $\mathfrak X$. All roots and weights are regarded as elements of 
the $r$-dimensional real vector space $\mathfrak X_\BR:=\mathfrak X\otimes \BR$.
\\ \indent
{\bf --}  $(\ ,\ )$ is a $W$-invariant symmetric non-degenerate bilinear form on $\mathfrak X_\BR$ and $s_i\in W$ is the reflection corresponding to $\ap_i$.
For any $\lb\in\mathfrak X_+$, let $\lb^*$ denote the highest weight of the dual $G$-module,
i.e., $\sfr(\lb)^*\simeq \sfr(\lb^*)$. The $\mu$-weight space of $\sfr(\lb)$ is denoted by $\sfr(\lb)_\mu$.
\\[.7ex]
We refer to \cite{vo} for standard results on root systems and representations of semisimple
algebraic groups.

\section{Recollections}     
\label{sect:recol}

\subsection{Horospherical varieties with a dense orbit}  \label{subs:horosph}
A $G$-variety $X$ is said to be {\it horospherical\/} if the stabiliser of any 
$x\in X$ contains a maximal unipotent subgroup of $G$. 
Following \cite{vp72},  affine 
horospherical varieties with a dense $G$-orbit are called $\eus S$-{\it varieties}.
Let $\mathfrak S$ be a finitely generated monoid in $\mathfrak X_+$ and 
$\{\lb_1,\dots,\lb_m\}$ the minimal set of generators of $\mathfrak S$. 
Let $v_{-\lb_i}\in \sfr(\lb_i^*)$ be a lowest weight
vector. Set $\bv=(v_{-\lb_1},\dots, v_{-\lb_m})$ and consider
\[
  \gc(\mathfrak S):=\ov{G{\cdot}\bv} \subset \sfr(\lb_1^*)\oplus\dots\oplus \sfr(\lb_m^*) .
\]
Clearly,  $\gc(\mathfrak S)$ is an $\eus S$-variety;  conversely, each 
$\eus S$-variety is obtained in this way   \cite{vp72}.
Write 
$\langle \mathfrak S\rangle$ for the linear span of $\mathfrak S$ in $\mathfrak X_\BR$
and set $\rk\mathfrak S=\dim_\BR \langle \mathfrak S\rangle$.
Let $L_\mathfrak S$ be the Levi subgroup such that $T\subset L_\mathfrak S$ and  the
roots of $L_\mathfrak S$ are those orthogonal to
$\lb_1,\dots,\lb_m$. Then
$P_\mathfrak S=L_\mathfrak S N_\mathfrak S$  is the standard parabolic subgroup, 
with unipotent radical $N_\mathfrak S\subset U$.  

\begin{thm}[\protect{\cite{vp72}}]  \label{thm:vp}
The affine variety $\gc(\mathfrak S)$  has the following properties:
\begin{itemize}
\item[\sf 1.] \ The algebra $\bbk[\gc(\mathfrak S)]$ is a multiplicity free $G$-module. More
precisely, 
$\bbk[\gc(\mathfrak S)]=\bigoplus_{\lb\in \mathfrak S}\sfr(\lb)$ and this decomposition is a multigrading, i.e.,  $\sfr(\lb)\sfr(\mu)=\sfr(\lb+\mu)$;
\item[\sf 2.] The $G$-orbits in $\gc(\mathfrak S)$ are in a one-to-one correspondence with 
the faces of the convex polyhedral cone in $\mathfrak X_\BR$ generated by $\mathfrak S$;
\item[\sf 3.] \ $\gc(\mathfrak S)$ is normal if and only if\/ $\BZ\mathfrak S\cap
\BQ^+\mathfrak S=\mathfrak S$;
\item[\sf 4.]  \ $\dim \gc(\mathfrak S)=\dim G/P_\mathfrak S + \rk\mathfrak S$.
\end{itemize}\end{thm}
\noindent
If $\mathfrak S=\BN\lb$, then we write $\gc(\lb), P_\lb,\dots$ in place of 
$\gc(\BN\lb), P_{\BN\lb}$,\dots.
The variety $\gc(\lb)$ is the closure of the
$G$-orbit of highest weight vectors in $\sfr(\lb^*)$. Such varieties are called $\mathsf{HV}$-{\it varieties}; they are always normal. Recall that a $G$-variety $X$ is {\it spherical}, if $B$
has a dense orbit in $X$. Since $B{\cdot}\bv$ is  dense in $\gc(\mathfrak S)$, all $\eus S$-varieties are spherical.
By \cite[Theorem\,10]{po86}),  a normal spherical variety has 
rational singularities and therefore is Cohen-Macaulay. In particular, if $\mathfrak S$ is 
a free monoid, then $\gc(\mathfrak S)$ has rational singularities.

\subsection{Generalities on $U'$-invariants}    
\label{subs:gen}
We recall some results of \cite{odno-sv} and thereby fix relevant notation.
We regard $\mathfrak X$ as a poset with respect to the
{\it root order\/} ``$\curle$''. This means that $\nu\curle \mu$ if $\mu-\nu$ is a 
non-negative integral linear combination of simple roots. 
For any $\lb\in\mathfrak X_+$, we fix a simple $G$-module $\sfr(\lb)$ and write
$\eus P(\lb)$ for the set of $T$-weights of $\sfr(\lb)$. Then $(\eus P(\lb),\curle)$ is a 
finite poset and $\lb$ is its  unique maximal element.
Let $e_i\in\ut =\Lie U$ be a root vector corresponding to $\ap_i\in\Pi$. Then $(e_1,\dots,e_r)$ is a
basis for $\Lie(U/U')$.

The subspace of $U'$-invariants in $\sfr(\lb)$ has a nice description.
Since $\sfr(\lb)^{U'}$ is acted upon by $B/U'$, it is $T$-stable. Hence
$\sfr(\lb)^{U'}=\bigoplus_{\mu\in \eus I_\lb}\sfr(\lb)^{U'}_\mu$, where 
$\eus I_\lb$ is a subset of $\eus P(\lb)$.

\begin{thm}[\protect {\cite[Theorem\,1.6]{odno-sv}}] \label{thm1.6}
Suppose that $\lb=\sum_{i=1}^r a_i\varpi_i\in \mathfrak X_+$. Then 

(1)\quad $\eus I_\lb=\{ \lb-\sum_{i=1}^r b_i\ap_i \mid 0\le b_i\le a_i \ \ \forall i\}$;

(2)\quad $\dim \sfr(\lb)^{U'}_\mu=1$ for all $\mu\in \eus I_\lb$, i.e., $ \sfr(\lb)^{U'}$ is a 
multiplicity free $T$-module;

(3)\quad A nonzero $U'$-invariant of weight $\lb-\sum_{i=1}^r a_i\ap_i$, say $\bof$, 
is a cyclic vector of the $U/U'$-module $\sfr(\lb)^{U'}$. That is, the vectors 
$\{ (\prod_{i=1}^r e_i^{b_i})(\bof) \mid 0\le b_i\le a_i  \ \ \forall i\}$ form a basis for 
$\sfr(\lb)^{U'}$.
\end{thm}

\noindent
It follows from  (1) and (2) that $\dim\sfr(\lb)^{U'}=
\prod_{i=1}^r(a_i+1)$.  In particular,  
$\dim\sfr(\varpi_i)^{U'}=2$. The weight spaces $\sfr(\varpi_i)_{\varpi_i}$ and
$\sfr(\varpi_i)_{\varpi_i-\ap_i}$ are one-dimensional, and we fix corresponding nonzero
weight vectors $f_i,\tilde f_i$ such that $e_i(\tilde f_i)=f_i$. That is, $\tilde f_i$
is a cyclic vector of $\sfr(\varpi_i)^{U'}$.

The biggest $\eus S$-variety corresponds to the monoid $\mathfrak S=\mathfrak X_+$.
Here 
\[
   \bbk[G/U]=\bbk[\gc(\mathfrak X_+)]=\bigoplus_{\lb\in \mathfrak X_+}\sfr(\lb) ,
\]
and the multiplicative structure of $\bbk[\gc(\mathfrak X_+)]$ together with
Theorem~\ref{thm1.6} imply 

\begin{thm}[cf. \protect {\cite[Theorem\,1.8]{odno-sv}}]    \label{thm1.8}
The  algebra of $U'$-invariants $\bbk[\gc(\mathfrak X_+)]^{U'}$ is  freely generated by
$f_1,\tilde f_1,\dots,f_r,\tilde f_r$. Therefore, any basis for the $2r$-dimensional
vector space $\bigoplus_{i=1}^r \mathsf R(\varpi_i)^{U'}$ yields a free generating
system for\/  $\bbk[\gc(\mathfrak X_+)]^{U'}$.
\end{thm}

\noindent
The algebra $\bbk[G/U]$ is sometimes
called the {\it flag algebra\/} for $G$, because it can be realized as the multi-homogeneous
coordinate ring of the flag variety $G/B$.
More generally, we have

\begin{thm}    \label{thm:some-fund}
If\/ $\mathfrak S$ is generated by some fundamental weights, say $\{\varpi_i\mid i\in M\}$, 
then any basis for $\bigoplus_{i\in M} \mathsf R(\varpi_i)^{U'}$ yields a free generating
system for\/  $\bbk[\gc(\mathfrak S)]^{U'}$.
\end{thm}
\begin{proof}
As in the proof of \cite[Theorem\,1.8]{odno-sv}, one observes that, for
$\lb=\sum_{i\in M}a_i\varpi_i$, the monomials 
$\{\prod_{i\in M}f_i^{b_i}\tilde f_i^{a_i-b_i} \mid 0\le b_i\le a_i \}$
form a basis for the space $\sfr(\lb)^{U'}$.
[Another way is to consider the natural embedding $\gc(\mathfrak S)\hookrightarrow
\gc(\mathfrak X_+)$ \cite{vp72}
and the surjective homomorphism $\bbk[\gc(\mathfrak X_+)]^{U'}
\to \bbk[\gc(\mathfrak S)]^{U'}$.]
\end{proof}

Given $\lb\in\mathfrak X_+$, we always consider  a basis for $\sfr(\lb)^{U'}$ generated by a cyclic vector and elements $e_i\in \g_{\ap_i}$, i.e., a basis
$\{ f_\mu \in \sfr(\lb)_\mu \mid  \mu\in \eus I_\lb\}$ such that
\[
    e_i(f_\mu)=\begin{cases}  f_{\mu+\ap_i}, &  \mu+\ap_i\in \eus I_\lb , \\
    0, &  \mu+\ap_i\not\in \eus I_\lb .  \end{cases}
\]  
However, for the fundamental $G$-modules $\sfr(\varpi_i)$, we write 
$f_i$ in place of $f_{\varpi_i}$ and $\tilde f_i$ in place of $f_{\varpi_i-\ap_i}$.

\subsection{Equidimensional morphisms and conical varieties}    \label{subs:equi}
Let $\pi:X\to Y$ be a dominant morphism of irreducible algebraic varieties. We say that
$\pi$ {\it is equidimensional at\/} $y\in Y$ if all irreducible components of $\pi^{-1}(y)$ are of
dimension $\dim X-\dim Y$. Then $\pi$ is said to be {\it equidimensional\/} if it is 
equidimensional at  any $y\in \pi(X)$. By a result of Chevalley \cite[Ch.\,5, n.5, Prop.\,3]{klod},
if $y=\pi(x)$ is a normal point, $\pi$ is equidimensional at $y$, and 
$\Omega\subset X$ is a neighbourhood of $x$, then
$\pi(\Omega)$ is a neighbourhood of $y$. Consequently, an equidimensional morphism to
a normal variety is open.

An affine variety $X$ is said to be {\it conical\/} if $\bbk[X]$ is $\BN$-graded, 
$\bbk[X]=\bigoplus_{n\ge 0} \bbk[X]_n$, and $\bbk[X]_0=\bbk$. Then the point $x_0$ 
corresponding to the maximal ideal $\bigoplus_{n\ge 1} \bbk[X]_n$ is called the
{\it vertex\/}. Geometrically, this means that $X$ is equipped with an action of the 
multiplicative group $\bbk^\times$ such that $\{x_0\}$ is the only closed $\bbk^\times$-orbit
in $X$.

\begin{lm}   \label{klod}
Suppose that both $X$ and $Y$ are  conical, and
$\pi:X\to Y$ is dominant and $\bbk^\times$-equivariant. 
(Then $\pi(x_0)=:y_0$ is the vertex in $Y$.)
If\/ $Y$ is normal and $\pi$ is equidimensional at 
$y_0$, then $\pi$ is onto and  equidimensional.
\end{lm}

\noindent
This readily follows from the above-mentioned result of Chevalley and
standard inequalities for the dimension of fibres.
\begin{rmk}   As $\mathfrak S$ lies in an open half-space of $\mathfrak X_\BR$,
taking a suitable $\BN$-specialisation of the multi-grading of $\bbk[\gc(\mathfrak S)]$
shows that $\gc(\mathfrak S)$ is conical and the origin in 
$\sfr(\lb_1^*)\oplus\dots\oplus \sfr(\lb_m^*)$ is its vertex.
This implies that $\gc(\mathfrak S)\md U'$ is conical, too.
We will apply the above lemma to the study of equidimensional quotient maps
$\pi:\gc(\mathfrak S)\to  \gc(\mathfrak S)\md U'$. It is important that such $\pi$ appears to be 
onto. 

The idea of applying Chevalley's result to the study of equidimensional quotients (by $U$)
is due to Vinberg and Gindikin \cite{vg08}.
\end{rmk}

\section{Actions of $U'$ on HV-varieties}    \label{sect:HV}

\noindent
Let $\gc(\lb)=\ov{G{\cdot}v_{-\lb}}\subset \sfr(\lb^*)$ be an \textsf{HV}-variety. The 
algebra $\bbk[\gc(\lb)]$ is $\BN$-graded and its component of degree $n$ is 
$\sfr(n\lb)$.  Since $\gc(\lb)$ is normal, $\gc(\lb)\md U'$ is normal, too.

\begin{thm}   \label{thm:HV-free}
$\gc(\lb)\md U'$ is an affine space if and only if  $\lb$ is a fundamental weight.
\end{thm}
\begin{proof}
1) \ Suppose that $\lb$ is not fundamental, i.e., 
$\lb=\cdots +a\varpi_i +b\varpi_j +\cdots$ with $a,b \ge 1$.

\textbullet \ 
If $i\ne j$, then $\sfr(\lb)^{U'}$ contains linearly independent vectors
 $f_\lb, f_{\lb-\ap_i}, f_{\lb-\ap_j}, f_{\lb-\ap_i-\ap_j}$ that occur in any minimal generating system, since $\bbk[\gc(\lb)]_1\simeq \sfr(\lb)$.
Using the relations $e_i(f_{\lb-\ap_i-\ap_j})=f_{\lb-\ap_j}$, etc., one easily verifies that  
\[
   p=f_\lb f_{\lb-\ap_i-\ap_j}-f_{\lb-\ap_i} f_{\lb-\ap_j}  
\]
is a $U$-invariant function on $\gc(\lb)$, of degree $2$. 
The only highest weight in degree $2$
is $2\lb$. Since the weight of $p$ is not $2\lb$, we must have $p\equiv 0$, and this is a 
non-trivial relation.

\textbullet \ 
If $i=j$, then the coefficient of $\varpi_i$ is at least $2$ and we consider vectors $f_\lb$, $f_{\lb-\ap_i}$, $f_{\lb-2\ap_i}\in 
\sfr(\lb)^{U'}$.  Then $\tilde p=2f_\lb f_{\lb-2\ap_i}-f_{\lb-\ap_i}^2$ is a $U$-invariant
function of degree 2 and weight $2(\lb-\ap_i)$, and this yields the relation
$\tilde p=0$ in $\bbk[\gc(\lb)]^{U'}$.

2) \ If $\lb=\varpi_i$, then $\dim \mathsf R(\varpi_i)^{U'}=2$ and
$\gc(\varpi_i)\md U'\simeq \BA^2$ by Theorem~\ref{thm:some-fund}.
\end{proof}

For an affine variety $X$, let $\ed X$ denote the minimal number of generators of $\bbk[X]$ 
and $\hd(X)$ the homological dimension of $\bbk[X]$. If $\bbk[X]$ is a graded 
Cohen-Macaulay algebra, then $\hd(X)=\ed X-\dim X$ \cite[Ch.\,IV]{serre}.

\begin{thm}   \label{thm:dim-HV}
If\/ $\lb=\sum_{i=1}^r a_i\varpi_i\in \mathfrak X_+$, then
\begin{itemize}
\item[\sf (i)]  \  $\dim \gc(\lb)\md U'= 1+\#\{j\mid a_j\ne0\}$;
\item[\sf (ii)] \  the graded algebra $\bbk[\gc(\lb)]^{U'}$ is generated by functions of 
degree one, i.e., by the space $\sfr(\lb)^{U'}$,  
and $\ed \gc(\lb)\md U'=\prod_{i=1}^r (a_i+1)$.
\end{itemize}
\end{thm}
\begin{proof}
(i)  Recall that $P_\lb=L_\lb N_\lb$ is the standard parabolic subgroup associated with $\gc(\lb)$ and 
the simple roots of 
$L_\lb$ are those orthogonal to $\lb$. Set $k=\#\{j\mid a_j\ne0\}$. Then 
$\srk L_\lb:=\rk(L_\lb,L_\lb)=\rk G-k$  and  $\dim \gc(\lb)=\dim N_\lb+1$. 
Since $U{\cdot}(\bbk v_{-\lb})$ is dense in $\gc(\lb)$,  
$U(L_\lb):=U\cap L_\lb$ is a generic stabiliser for the $U$-action on $\gc(\lb)$. 
By \cite[Lemma\,2.5]{odno-sv}, the minimal dimension of stabilisers for the $U'$-action on
$\gc(\lb)$ equals  $\dim (U(L_\lb)\cap U')=\dim U(L_\lb)-\srk L_\lb$.
Consequently,
\begin{multline*}
   \dim \gc(\lb)\md U'=\dim \gc(\lb)-\dim U'+\min_{x\in \gc(\lb)}\dim U'_x=  \\
   =\dim N_\lb+1- (\dim U-\rk G)+(\dim U(L_\lb)-\srk L_\lb)=1+\rk G-\srk L_\lb=1+k .
\end{multline*}

(ii)   By Theorem~\ref{thm1.6},   $\dim \mathsf R(\lb)^{U'}=\prod_{i=1}^r (a_i+1)$, which 
shows that $\ed \gc(\lb)\md U' \ge \prod_{i=1}^r (a_i+1)$.
Therefore, it suffices to prove that the graded algebra 
$\bbk[\gc(\lb)]^{U'}$ is generated by elements of degree $1$.
The weights of $U'$-invariants of degree $n$ are
\[
   \eus I_{n\lb}=\{n\lb-\sum_i b_i\ap_i \mid b_i=0,1,\dots, na_i \} .
\]
In particular, 
\[
   \eus I_\lb=\{\lb-\sum_i b_i\ap_i \mid b_i=0,1,\dots, a_i \} .
\]
Obviously, each element of $\eus I_{n\lb}$ is a sum of $n$ elements of $\eus I_{\lb}$.
Since  $\mathsf R(n\lb)^{U'}$ is a multiplicity free $T$-module, this space
is spanned by products of $n$ elements of $\mathsf R(\lb)^{U'}$.
\end{proof}

\begin{cl}
We have $\hd(\gc(\lb)\md U')=\prod_{i=1}^r(1+a_i) - 1-\#\{j\mid a_j\ne0\}$.
Therefore, 

\textbullet \quad $\hd(\gc(\lb)\md U')=0$ if and only if $\lb$ is fundamental;

\textbullet \quad $\hd(\gc(\lb)\md U')=1$ if and only if $\lb=\varpi_i+\varpi_j$ or $2\varpi_i$.
\end{cl}\begin{proof}
As it was mentioned above, the {\sf HV}-varieties have rational singularities. 
In view of \cite[Theorem\,2.3]{odno-sv}, $\gc(\lb)\md U'$ also has rational singularities
and in particular is Cohen-Macaulay. Hence 
$\hd(\gc(\lb)\md U')=\ed \gc(\lb)\md U' - \dim \gc(\lb)\md U'$.
\end{proof}

\begin{rmk}   \label{rmk:toric}
1) \ As above, $k=\rk G-\srk L_\lb$
and hence $\dim \gc(\lb)\md U'=k+1$.
Another consequence of Theorems~\ref{thm1.6}  and \ref{thm:dim-HV} is that 
$\gc(\lb)\md U'$ is a toric variety with respect to $\bbk^\times \times T$, where 
$\bbk^\times$ acts on $\sfr(\lb^*)$ (and hence on $\gc(\lb)$) by homotheties. 
Note that the $T$-action on 
$\gc(\lb)\md U'$ has a non-effectivity kernel of dimension $\rk G-k$.
The quotient morphism
$\pi_{\gc(\lb), U'}$ has the following description. 
Let $\Ann(\sfr(\lb)^{U'})$ be the annihilator of $\sfr(\lb)^{U'}$ in $\sfr(\lb^*)$.
Then $(\sfr(\lb)^{U'})^*=\sfr(\lb^*)/\Ann(\sfr(\lb)^{U'})$ and $\pi_{\gc(\lb), U'}$ is the restriction 
to $\gc(\lb)$ of the projection $\sfr(\lb^*) \to (\sfr(\lb)^{U'})^*$.
Thus, 
$\gc(\lb)\md U'$ is embedded in the vector space $(\sfr(\lb)^{U'})^*$.
Consequently, $\mathbb P(\gc(\lb)\md U')\subset \BP((\sfr(\lb)^{U'})^*)$  is a normal 
toric variety with respect to $T$. As is well-known, a projective toric $T$-variety
can be described via a convex polytope in $\mathfrak X_\BQ$ \cite[5.8]{danilov}.
The polytope corresponding to $\mathbb P(\gc(\lb)\md U')$ is the convex hull of
$\eus I_\lb$. It is a $k$-dimensional parallelepiped, in particular, a simple polytope.
It follows that the corresponding complete fan is simplicial.  Therefore the complex cohomology of
$\mathbb P(\gc(\lb)\md U')$ satisfies  Poincar\'e duality and
has a number of other good properties, see \cite[\S\,14]{danilov}.

2) \ Along with the toric structure (i.e., a dense $T$-orbit), the projective variety $\mathbb P(\gc(\lb)\md U')$ also has 
a dense orbit of the commutative unipotent group $U/U'$.
\end{rmk}

\section{Actions of $U'$ on arbitrary $\eus S$-varieties}    
\label{sect:S}

\noindent
Let $\gc(\mathfrak S)$ be an $\eus S$-variety. 
In this section, we answer the following questions:

{\bf --}  When is $\gc(\mathfrak S)\md U'$ an affine space?

{\bf --} Suppose that $\gc(\mathfrak S)\md U'$ is an affine space. When is 
$\pi_{\gc(\mathfrak S),U'}$ equidimensional?

We begin with a formula for $\dim\gc(\mathfrak S)\md U'$, which generalises
Theorem~\ref{thm:dim-HV}(i).

\begin{prop}    \label{prop:dim-mod-U'}
$\dim \gc(\mathfrak S)\md U'=\rk \mathfrak S+ (\rk G-\srk L_\mathfrak S)$.
\end{prop}\begin{proof}
By Theorem~\ref{thm:vp}, $\dim \gc(\mathfrak S)=\dim N_\mathfrak S+\rk \mathfrak S$
and $\dim \gc(\mathfrak S)\md U=\rk \mathfrak S$. This readily implies that 
$U(L_\mathfrak S):=U\cap L_\mathfrak S$ is a generic stabiliser for the $U$-action on
$\gc(\mathfrak S)$.  By \cite[Lemma\,2.5]{odno-sv}, the minimal dimension of stabilisers for the $U'$-action on
$\gc(\mathfrak S)$ equals  $\dim (U(L_\mathfrak S)\cap U')=
\dim U(L_\mathfrak S)-\srk L_\mathfrak S$.
Consequently,
\begin{multline*}
   \dim \gc(\mathfrak S)\md U'=\dim \gc(\mathfrak S)-\dim U'+\min_{x\in \gc(\mathfrak S)}\dim U'_x=  \\
   =\dim N_\mathfrak S+\rk\mathfrak S- (\dim U-\rk G)+(\dim U(L_\mathfrak S)-\srk L_\mathfrak S)=
   \rk\mathfrak S+(\rk G-\srk L_\mathfrak S) .
\end{multline*}
Here we use the fact that  $U$ is a semi-direct product of $N_\mathfrak S$ and
$U(L_\mathfrak S)$.
\end{proof}

\begin{rema}
Note that $\rk \mathfrak S\le \rk G-\srk L_\mathfrak S$, and the equality here is equivalent 
to the fact that  the space $\langle\mathfrak S\rangle$ has a basis that consists of
fundamental weights.
\end{rema}

\begin{thm}   \label{thm:S-var}
Let $\mathfrak S\subset \mathfrak X_+$ be an arbitrary finitely generated monoid.  
Then $\gc(\mathfrak S)\md U'$ is an affine space if and only if $\mathfrak S$ is 
generated by fundamental weights.
\end{thm}
\begin{proof} 
1)  \ Suppose that $\gc(\mathfrak S)\md U'$ is an affine space.
If $\lb$ is a generator of $\mathfrak S$, 
then any generating system of $\bbk[\gc(\mathfrak S)]^{U'}$ contains a basis for
$\sfr(\lb)^{U'}$. Arguing as in the proof of Theorem~\ref{thm:HV-free}, we conclude that $\lb$ must be a fundamental weight.  [Another way is to use Proposition~\ref{prop:dim-mod-U'} and 
the inequality $\dim \gc(\mathfrak S)\md U'\ge 2\rk\mathfrak S$.]

2) The converse is contained in Theorem~\ref{thm:some-fund}.
 \end{proof}

\noindent
In the rest of this section, we only consider monoids generated by  fundamental weights.
Fix a numbering of the simple roots  (fundamental weights). 
For any $M\subset\{1,2,\dots,r\}$, let $\gc(M)$ denote the $\eus S$-variety 
corresponding to the monoid $\mathfrak S=\sum_{i\in M} \BN \varpi_i$.
Our aim is to characterise the subsets $M$ having the property that
$\pi_{U'}:  \gc(M)\to \gc(M)\md U'$ is equidimensional. The origin 
(vertex) is the only $G$-fixed point
of $\gc(M)$ and the corresponding fibre
of $\pi_{U'}$ (the {\it null-cone}) is denoted by
$\fN_{U'}(M)$.  

Recall that  $\bbk[\gc(M)]$ is a graded Cohen-Macaulay ring and $\bbk[\gc(M)]^{U'}$ 
is a polynomial algebra freely generated by  $\{f_i,\tilde f_i \mid i\in M\}$ 
(Theorem~\ref{thm:some-fund}). Therefore, $\pi_{U'}$ is equidimensional 
{\sl if and only if\/} the functions $\{f_i,\tilde f_i \mid i\in M\}$ form a regular sequence in 
$\bbk[\gc(M)]$ 
{\sl if and only if\/}  $\dim \fN_{U'}(M)=\dim \gc(M)- 2(\#M)$
\cite[\S\,17]{gerry}.

\begin{df}    \label{def:sparse}
A subset $M\subset \{1,\dots,r\}$ is said to be {\it sparse\/}, if 
1) the roots $\ap_i$ with $i\in M$ are pairwise orthogonal,
i.e., disjoint in the Dynkin diagram; 2) there are no $i,j\in M$ and no $k\not\in M$ such that 
$(\ap_k,\ap_i)<0$ and $(\ap_k,\ap_j)<0$, i.e., 
$\ap_k$ is adjacent to both $\ap_i$ and $\ap_j$. 
\\ 
Accordingly, we  say that a certain set of fundamental weights (simple roots) is {\it sparse}.
\end{df}

Clearly, if $M$ is sparse and $J\subset M$, then $J$ is also sparse.
\begin{lm}    \label{lm:mult1}
Let $\ap_{i_1},\dots,\ap_{i_l}$ be a sequence of different simple roots such that
$\ap_{i_j},\ap_{i_{j+1}}$ are adjacent for $j=1,2,\dots,l-1$). Then $\mu:=\varpi_{i_1}-\sum_{j=1}^l \ap_{i_j}$ is a weight of $\sfr(\varpi_{i_1})$ and
$\dim \sfr(\varpi_{i_1})_\mu=1$.
\end{lm}
\begin{proof}
The first assertion is easily proved by induction on $l$. The second assertion follows from 
\cite[Prop.\,2.2]{berzel}
\end{proof}

\begin{thm}   \label{thm:eq-sparse1}
If the quotient
$\pi_{U'}:  \gc(M)\to \gc(M)\md U'$ is equidimensional, then $M$ is sparse.
\end{thm}\begin{proof}
As we already know, $\bbk[\gc(M)]^{U'}$ is freely generated by the functions 
$\{f_i,\tilde f_i \mid i\in M\}$. 
Assuming that $M$ is not sparse, we point out certain relations in $\bbk[\gc(M)]$, which show 
that these free generators do not form a regular sequence.
There are two possibilities for that.

\textbullet \quad Suppose first that  $\ap_i$ and $\ap_j$ are adjacent simple roots
for some $i,j\in M$. 
Then $\lb_{ij}:=\varpi_i+\varpi_j-\ap_i-\ap_j$ is dominant.
Consider upper parts of the Hasse diagrams of weight posets for $\mathsf R(\varpi_i)$ and $\mathsf R(\varpi_j)$:

$\mathsf R(\varpi_i)$:   
\begin{picture}(300,40)(10,8)
\multiput(30,12)(60,0){3}{\circle{6}}
\multiput(33,12)(60,0){2}{\line(1,0){54}}
\put(153,12){\line(1,0){20}}
\put(25,20){$\varpi_i$}
\put(25,-5){$f_i$}
\put(85,-5){$\tilde f_i$}
\put(145,-5){$p_i$}
\put(78,20){$\varpi_i{-}\ap_i$}
\put(133,20){$\varpi_i{-}\ap_i{-}\ap_j$}
\put(180,12){\dots}
\end{picture}

$\mathsf R(\varpi_j)$:   
\begin{picture}(300,45)(10,8)
\multiput(30,12)(60,0){3}{\circle{6}}
\multiput(33,12)(60,0){2}{\line(1,0){54}}
\put(153,12){\line(1,0){20}}
\put(25,20){$\varpi_i$}
\put(25,-5){$f_j$}
\put(85,-5){$\tilde f_j$}
\put(145,-5){$p_j$}
\put(78,20){$\varpi_j{-}\ap_j$}
\put(133,20){$\varpi_j{-}\ap_i{-}\ap_j$}
\put(180,12){\dots}
\end{picture}

\vskip2ex

\noindent In these figures, each node depicts a weight space, and we put the weight 
over the node and a weight vector under the node. 
There can be other edges incident to the node $\varpi_i-\ap_i$ (if there exist other simple
roots adjacent to $\ap_i$), but we do not need them.
By Lemma~\ref{lm:mult1}, the weight spaces
$\sfr(\varpi_i)_{\varpi_i}$, $\sfr(\varpi_i)_{\varpi_i-\ap_i}$, and $\sfr(\varpi_i)_{\varpi_i-\ap_i-\ap_j}$ are one-dimensional. 
Here $f_i,\tilde f_i$, and $p_i$ are normalised such that $e_i(\tilde f_i)=f_i$ and
$e_j(p_i)=\tilde f_i$;
and likewise for $\mathsf R(\varpi_j)$. 
Note also that $e_i(p_i)=0$, since $\varpi_i-\ap_j$ is not a weight of
$\sfr(\varpi_i)$. It is then easily seen that
\[
  f_i\otimes p_j-\tilde f_i\otimes \tilde f_j+p_i\otimes f_j
\]   
is a $U$-invariant of weight 
$\lb_{ij}$ in $\mathsf R(\varpi_i)\otimes \mathsf R(\varpi_j)$.
However, only the Cartan component of
$\mathsf R(\varpi_i)\otimes \mathsf R(\varpi_j)$ survives in the algebra $\bbk[\gc(M)]$, i.e.,
in the product $\mathsf R(\varpi_i){\cdot}\mathsf R(\varpi_j)$.
Consequently, $f_ip_j-\tilde f_i \tilde f_j+p_i f_j=0$ in $\bbk[\gc(M)]$. 
This means that $(f_i,f_j,\tilde f_i,\tilde f_j)$ is not a regular sequence
in $\bbk[\gc(M)]$.

\textbullet \quad Yet another possibility is that there are $k\not\in M$ and
$i,j\in M$ such that  $\ap_k$  is adjacent to  both $\ap_i$ and $\ap_j$. 
Here one verifies that 
$\tilde\lb_{ij}:=\varpi_i+\varpi_j-\ap_i-\ap_k-\ap_j$ is dominant. In this situation, we need
larger fragments of the weight posets:

$\mathsf R(\varpi_i)$:   
\begin{picture}(300,40)(10,8)
\multiput(30,12)(60,0){4}{\circle{6}}
\multiput(33,12)(60,0){3}{\line(1,0){54}}
\put(213,12){\line(1,0){20}}
\put(25,20){$\varpi_i$}
\put(25,-5){$f_i$}
\put(85,-5){$\tilde f_i$}
\put(145,-5){$p_i$}
\put(205,-5){$q_i$}
\put(78,20){$\varpi_i{-}\ap_i$}
\put(133,20){$\varpi_i{-}\ap_i{-}\ap_k$}
\put(200,20){$\varpi_i{-}\ap_i{-}\ap_k{-}\ap_j$}
\put(240,12){\dots}
\end{picture}

$\mathsf R(\varpi_j)$:   
\begin{picture}(300,45)(10,8)
\multiput(30,12)(60,0){4}{\circle{6}}
\multiput(33,12)(60,0){3}{\line(1,0){54}}
\put(213,12){\line(1,0){20}}
\put(25,20){$\varpi_i$}
\put(25,-5){$f_j$}
\put(85,-5){$\tilde f_j$}
\put(145,-5){$p_j$}
\put(205,-5){$q_j$}
\put(78,20){$\varpi_j{-}\ap_j$}
\put(133,20){$\varpi_j{-}\ap_j{-}\ap_k$}
\put(200,20){$\varpi_j{-}\ap_j{-}\ap_k{-}\ap_i$}
\put(240,12){\dots}
\end{picture}

\vskip2ex  

\noindent
Here all the weight spaces are one-dimensional by Lemma~\ref{lm:mult1}, and
we follow the same conventions as above. Additionally, we assume that
$e_j(q_i)=p_i$.  
Note that  $e_k(q_i)=0$ and $e_i(q_i)=0$, since neither 
$\varpi_i{-}\ap_i{-}\ap_j$ nor $\varpi_i{-}\ap_k{-}\ap_j$ is  
a weight of $\mathsf R(\varpi_i)$. (And likewise for  $\mathsf R(\varpi_j)$.)
Then
   $ f_i\otimes q_j-\tilde f_i\otimes p_j+p_i\otimes \tilde f_j-q_i\otimes f_j$
is a $U$-invariant of weight $\tilde\lb_{ij}$, and hence 
\beq   \label{eq:sootnoshenie}
    f_iq_j-\tilde f_ip_j+p_i\tilde f_j-q_if_j=0
\eeq    
in $\bbk[\gc(M)]$ for the same reason as above.
This again implies that $(f_i,f_j,\tilde f_i,\tilde f_j)$ is not a regular sequence
in $\bbk[\gc(M)]$.
\end{proof}

\begin{ex}
Let $\g=\mathfrak{sl}_4$ and $M=\{1, 3\}$ in the usual numbering of $\Pi$. 
Then $\dim \mathsf R(\varpi_1)=\dim \mathsf R(\varpi_3)=4$ and 
$\dim\gc(M)=7$. 
In this case, the above $4$-node fragments provide the whole weight posets.
Therefore,  $\sfr(\varpi_1)=\langle f_1,\tilde f_1, p_1,q_1\rangle$, 
$\sfr(\varpi_3)=\langle f_3,\tilde f_3, p_3,q_3\rangle$,
and \eqref{eq:sootnoshenie} with $(i,j)=(1,3)$ is the equation of the hypersurface
$\gc(M)$. Since $\dim\gc(M)\md U'=4$ and
$\fN_{U'}(M)\supset \langle p_1,q_1, p_3,q_3\rangle$, the morphism $\pi_{U'}$ is not
equidimensional.
\end{ex}

To prove the converse to Theorem~\ref{thm:eq-sparse1}, we need some 
preparations. Recall that the partial order ``$\curle$'' is defined in \ref{subs:gen}. We
also write $\nu\prec\mu$ if $\nu \curle \mu$ and $\mu\ne\nu$.

\begin{lm}   \label{lm:dlina}
Suppose that $M$ is sparse and  $w\in W$ has the property that
$w(\varpi_i) \prec \varpi_i-\ap_i$ for all $i\in M$.  Then $\ell(w) \ge 2{\cdot}\#(M)$.
\end{lm}\begin{proof}
Since $w(\varpi_i) \prec \varpi_i$, \ any reduced decomposition of $w$ contains $s_i$.
Furthermore, since $w(\varpi_i) \prec \varpi_i-\ap_i$, there exists a node $i'$ adjacent to 
$i$ such that $w(\varpi_i) \curle  \varpi_i-\ap_i-\ap_{i'}$. Therefore, $w$ must also contain the
reflection $s_{i'}$. Because $M$ is sparse, all the reflections $\{s_i,s_{i'}\mid i\in M\}$ are different.
Thus, $\ell(w)\ge 2{\cdot}\#(M)$.
\end{proof}

For any   $I\subset\Pi$, we consider the following objects. 
Let $P_I=L_IN_I$ be the standard parabolic subgroup of $G$.
Here $L_I$ is the Levi subgroup whose set of simple roots is $I$ and 
$N_I$ is the unipotent radical of $P_I$. Then $P_I^-=L_IN_I^-$ is the opposite parabolic
subgroup of $G$.
We also need the factorisation
\[
     W=W^I\times W_I ,
\]
where $W_I$ is the  subgroup generated by $\{s_i\mid \ap_i\in I\}$ and 
$W^I$ is the set of repre\-sen\-ta\-tives of {minimal} length for $W/W_I$ \cite[1.10]{hump}.
It is also true that  $W^I=\{w\in W\mid w(\ap_i)\in\Delta^+ \ \ \forall \ap_i\in I\}$ 
\cite[5.4]{hump}. If $I=\{\ap\in \Pi \mid (\ap,\lb)= 0\}$ for some $\lb\in \mathfrak X_+$, then
we write $P_\lb$, $W_\lb$, $W^\lb$, etc. 

For each $w\in W$,
we fix a representative, $\dot w$, in $N_G(T)$.
As is well-known, the $U$-orbits in $G/P_I^-$ can be parametrised by $W^I$, and letting
$\co(w)=U\dot wP_I^-\subset G/P_I^-$ ($w\in W^I$), we have  
$G/P_I^-=\sqcup_{w\in W^I}\co(w)$
and $\codim \co(w)=\ell(w)$.

\begin{thm}    \label{thm:eq-sparse2}
If $M\subset\{1,\dots,r\}$ is sparse, then the quotient 
$\pi_{U'}:  \gc(M)\to \gc(M)\md U'$ is equidimensional.
\end{thm}\begin{proof}
Set $m=\#M$ and $I=\Pi\setminus \{\ap_i\mid i\in M\}$.
Consider $\bv=\sum_{i\in M} v_{-\varpi_i}\in\bigoplus_{i\in M}\mathsf R(\varpi_i^*)$.
As explained in Subsection~\ref{subs:horosph}, 
then $\gc(M)\simeq \ov{G{\cdot}\bv}$ and
$\dim \gc(M)=\dim G/P_I^-+m$. We also have  $\dim \gc(M)\md U'=2m$. 
Therefore, our goal is to prove that $\dim\fN_{U'}(M) \le \dim G/P_I^--m$.

Set $V=\ov{T{\cdot}\bv}=\bigoplus_{i\in M} \bbk v_{-\varpi_i}$. It is an $m$-dimensional subspace of
$\bigoplus_{i\in M}\mathsf R(\varpi_i^*)$, which  is contained in $\gc(M)$ and is $P_I^-$-stable. 
Recall that $G\times_{P_I^-} V$ is a homogeneous vector bundle on $G/P_I^-$.
A typical element of it is denoted by $g\ast v$, where $g\in G$ and $v=\sum_{i\in M} v_i\in V$.
Our main tool for estimating $\dim\fN_{U'}(M)$ is the following diagram:
\[
\begin{array}{ccc} G\times_{P_I^-} V & \stackrel{\tau}{\longrightarrow} &
\gc(M) \\
\Big\downarrow\vcenter{%
\rlap{$\phi$}} & & \Big\downarrow\vcenter{%
\rlap{$\pi_{U'}$}} \\
G/P_I^- & & \gc(M)\md U'  \end{array}
\]
where $\phi(g\ast v):=gP_I^-$ and $\tau(g\ast v):=g{\cdot}v$. 
Note that $\fN_{U'}(M)$ is $B$-stable, and hence so is $\tau^{-1}(\fN_{U'}(M))$.
It is easily seen that the morphism $\tau$ 
is birational and therefore it is an equivariant resolution of singularities
of $\gc(M)$.

Let $n\in U$ and $w\in W^I$. As $\bbk[\gc(M)]^{U'}$ is generated by 
$\{f_i,\tilde f_i \mid i\in M\}$, we have
\beq    \label{eq:star}
   \phi^{-1}(n\dot wP_I^-)\cap \tau^{-1}(\fN_{U'}(M))=\{n\dot w\ast v \mid f_i(n\dot w{\cdot}v)=0,\ 
   \tilde f_i(n\dot w{\cdot}v)=0 \quad \forall i\in M\} .
\eeq
Here  $f_i$ (resp. $\tilde f_i$)
is regarded as the coordinate  of $v_{-\varpi_i}\in \mathsf R(\varpi_i^*)$ (resp. 
 $v_{-\varpi_i+\ap_i}\in \mathsf R(\varpi_i^*)$).  
Note that $f_i(n\dot w{\cdot}v)$ depends only on the component $v_i$ of $v$, and $v_i$ is 
proportional to  $v_{-\varpi_i}$. 
Let us simplify condition \eqref{eq:star}.
Since $f_i$ is actually a $U$-invariant, we have $f_i(n\dot w{\cdot}v_i)=f_i(\dot w{\cdot}v_i)$. Next, 
$\tilde f_i$ is invariant with respect to a subgroup of codimension 1 in $U$.
Namely, consider the decomposition 
$U=U^{\ap_i}U_{\ap_i}\simeq U^{\ap_i}\times U_{\ap_i}$, where 
$U_{\ap_i}$ is the root subgroup and $U^{\ap_i}$ is the unipotent radical of the minimal
parabolic subgroup associated with $\ap_i$.
If $n_i\in U_{\ap_i}$  and $\tilde n\in U^{\ap_i}$, 
then $\tilde n{\cdot}\tilde f_i=\tilde f_i$ and 
$n_i^{-1}{\cdot}\tilde f_i=\tilde f_i+c_if_i$ for some
$c_i=c_i(n_i)\in \bbk$. Hence for $n=\tilde nn_i\in U$, we have
\[
  \tilde f_i(n\dot w{\cdot}v_i)=\tilde f_i(n_i\dot w{\cdot}v_i)=(n_i^{-1}{\cdot}\tilde f_i)(\dot 
  w{\cdot}v_i)=\tilde f_i(\dot w{\cdot}v_i)+f_i(\dot w{\cdot}v_i)c_i \, .
\]  
Therefore, \eqref{eq:star}  reduces to the following:
\beq  \label{eq:star2}
   \phi^{-1}(n\dot wP_I^-)\cap \tau^{-1}(\fN_{U'}(M))=\{n\dot w\ast v \mid  f_i(\dot w{\cdot}v_i)=0,\ 
   \tilde f_i(\dot w{\cdot}v_i)=0 \quad \forall i\in M\} .
\eeq
Thus, the dimension of this intersection does not depend on $n\in U$;  it depends
only on $w\in W^I$, i.e., on $\co(w)\subset G/P_I^-$.
We can make \eqref{eq:star2} more precise by using the partition of
$\gc(M)$ into (finitely many) $G$-orbits. For any subset $J\subset M$, let  $\bv_J=
\sum_{i\in J}v_{-\varpi_i}\in V$. Then $\{\bv_J \mid J\subset M\}$ is a complete set of 
representatives of the $G$-orbits in $\gc(M)$ (Theorem~\ref{thm:vp}(2)). 
Set $\overset{\circ}{V}_J=G{\cdot}\bv_J\cap V=
T{\cdot}\bv_J$. It is an open subset of a ($\#J$)-dimensional vector space. 
Then
\[
   \phi^{-1}(n\dot wP_I^-)\cap \tau^{-1}(\fN_{U'}(M)\cap G{\cdot}\bv_J)=\{n\dot w\ast v \mid 
   v\in \overset{\circ}{V}_J,  \  f_i(\dot w{\cdot}v_i)=0,\ 
   \tilde f_i(\dot w{\cdot}v_i)=0 \ \forall i\in M\} .
\]
This set is non-empty if and only if   $\dot w{\cdot}v_{-\varpi_i}$ has the trivial projection to
$\langle v_{-\varpi_i}, v_{-\varpi_i+\ap_i}\rangle \subset \sfr(\varpi_i^*)$ for all $i\in J$, i.e.,
$w(\varpi_i)\prec \varpi_i-\ap_i$ for all $i\in J$.  In this case the dimension of
this set equals $\dim \overset{\circ}{V}_J=\#J$. 
Consequently, if $\phi^{-1}(\co(w))\cap \tau^{-1}(\fN_{U'}(M)\cap G{\cdot}\bv_J)\ne \varnothing$,
then 
\begin{gather*}
\text{$w(\varpi_i)\prec \varpi_i-\ap_i$ \ for all \ $i\in J$ and} 
\\
   \dim\Bigl(\phi^{-1}(\co(w))\cap \tau^{-1}(\fN_{U'}(M)\cap G{\cdot}\bv_J)\Bigr)=\#J+ \dim \co(w) .
\end{gather*}
By Lemma~\ref{lm:dlina}, $\ell(w)\ge 2{\cdot}\#J$. 
Therefore, 
\begin{multline*}
    \dim\Bigl(\phi^{-1}(\co(w))\cap \tau^{-1}(\fN_{U'}(M)\cap G{\cdot}\bv_J)\Bigr)= \\
    \# J - \codim \co(w)+ \dim G/P_I^-=\# J -\ell(w)+\dim G/P_I^- 
    \le \dim G/P_I^- -\#J .
\end{multline*}
This is an upper bound for the dimension of the pullback in $G\times_{P_I^-}V$
of  a subset of $\fN_{U'}(M)$. 
If $\bv_J$ is not generic, i.e., $J\ne M$, then $\dim \tau^{-1}(\bv_J)> 0$ and the actual
subset of $\fN_{U'}(M)$ has smaller dimension.
More precisely,  set $\tilde I=\{\ap_i \mid i\not\in J\}$. Then $\tilde I\supset I$ and
$\tau^{-1}(\bv_J)\simeq P_{\tilde I}^-/P_I^-$. Since $\srk(L_{\tilde I})=
\srk(L_I)+(m-\#J)$, we have $\dim \tau^{-1}(\bv_J)\ge m-\#J$.
Thus,  for all $w\in W^I$ and $J\subset M$, we have
\begin{multline*}
   \dim\Bigl(\tau\bigl(\phi^{-1}(\co(w))\bigr)\cap \fN_{U'}(M)\cap G{\cdot}\bv_J\Bigr) \le \\
     \dim G/P_I^- -\#J- (m-\#J)=\dim G/P_I^- -m ,
\end{multline*}
and therefore $\dim \fN_{U'}(M)\le \dim G/P_I^- -m$.
\end{proof}
\begin{rmk}
A ``dual''  approach is to consider the $P_I$-stable subspace 
$\tilde V=\bigoplus_{i\in M} \bbk v_{\varpi^*_i}
\subset \bigoplus_{i\in M}\mathsf R(\varpi_i^*)$ and the map
$G\times_{P_I}\tilde V\to \gc(M)$. Then one has to work with $U_-$-orbits in $G/P_I$
and $U_-$-invariants in $\bbk[\gc(M)]$, but all dimension estimates  remain the same. 
Such an approach is realised in \cite[Sect.\,5]{aura}, where
the equidimensionality problem is considered for the actions of $U$ 
on $\eus S$-varieties.
\end{rmk}

Combining Theorems~\ref{thm:S-var}, \ref{thm:eq-sparse1}, and \ref{thm:eq-sparse2}, 
we obtain the general criterion:

\begin{thm}   \label{thm:main-eq-S}
For a finitely generated monoid $\mathfrak S\subset \mathfrak X_+$, the following
conditions are equivalent:
\begin{itemize}
\item[\sf (i)] \   $\gc(\mathfrak S)\md U'$ is an affine space and 
$\pi_{\gc(\mathfrak S),U'}:  \gc(\mathfrak S)\to \gc(\mathfrak S)\md U'$ is equidimensional;
\item[\sf (ii)] \  $\mathfrak S$ is generated by a sparse set of fundamental weights.
\end{itemize}
\end{thm}

\section{Equidimensional quotients by $U'$}    \label{sect:Eq}

\noindent
In this section, the quotient morphism for the $\eus S$-variety $\gc(\mathfrak S)$
will be denoted by $\pi_{\mathfrak S,U'}$. Similarly, for the \textsf{HV}-variety $\gc(\lb)$, 
we use notation $\pi_{\lb,U'}$. Our goal is to characterise the monoids  $\mathfrak S$ such 
that  $\pi_{\mathfrak S, U'}:  \gc(\mathfrak S)\to \gc(\mathfrak S)\md U'$ is equidimensional
(i.e., without assuming that $\gc(\mathfrak S)\md U'$ is an affine space). We assume 
that $U'\ne \{1\}$, i.e., $G$ is not a product of several $SL_2$.

First, we consider the case of {\sf HV}-varieties. 

\begin{thm}  \label{thm:eq-HV}
For any $\lb\in \mathfrak X_+$, the null-cone $\fN_{U'}(\gc(\lb))$ is of codimension 
$2$ in $\gc(\lb)$.
\end{thm}\begin{proof}  
As in the proof of Theorem~\ref{thm:eq-sparse2}, we work with the diagram
\[
\begin{array}{ccc} G\times_{P_\lb^-} V & \stackrel{\tau}{\longrightarrow} &
\gc(\lb) \\
\Big\downarrow\vcenter{%
\rlap{$\phi$}} & & \Big\downarrow\vcenter{%
\rlap{$\pi_{\lb,U'}$}} \\
G/P_\lb^- & & \gc(\lb)\md U' , \end{array}
\]
where $V=\bbk v_{-\lb}$, $\phi(g\ast v):=gP_\lb^-$ and $\tau(g\ast v):=g{\cdot}v$.
Note that $P_\lb^-$ is just the stabiliser of the line $V\subset \sfr(\lb^*)$.
For simplicity, we write $\fN_{U'}(\lb)$ in place of $\fN_{U'}(\gc(\lb))$.

Since $\fN_{U'}(\lb)$ is $U$-stable, $\phi(\tau^{-1}(\fN_{U'}(\lb)))$ is a union of $U$-orbits. 
Recall that $\bbk[\gc(\lb)]^{U'}$ is generated by the space $\sfr(\lb)^{U'}$, and the 
corresponding set of $T$-weights is $\eus I_\lb$.

We point out a $w\in W^\lb$ such that the 
$U$-orbit $\co(w)\subset G/P_\lb^-$ is of codimension $2$
and $\phi^{-1}(\co(w))\subset \tau^{-1}(\fN_{U'}(\lb))$.  
Suppose that $(\lb,\ap_1^\vee)=a_1\ge 1$ and $\ap_1$ is a simple root of a simple 
component of $G$ of rank $\ge 2$.
Let $\ap_2$ be a simple root adjacent to $\ap_1$ in the Dynkin diagram. 
Take $w=s_2s_1$. Regardless of the value of $(\lb,\ap_2)$, it is true that
$w\in W^\lb$ and $\ell(w)=2$.  We have
\[
   s_2s_1(\lb)=\lb-a_1\ap_1-(a_2-a_1(\ap_1,\ap_2^\vee))\ap_2\curle 
   \lb-a_1\ap_1-(a_1+a_2)\ap_2, 
\]
where $a_2=(\lb,\ap_2^\vee)$. 
Hence $s_2s_1(\lb)\not\in \eus I_\lb$. It follows that  
$\dot{s_2}\dot{s_1}(v_{-\lb})\in \fN_{U'}(\lb)$
and 
\[
      \tau(\phi^{-1}(\co(w)))=U{\cdot}(\dot{s_2}\dot{s_1}(V))\in \fN_{U'}(\lb) .
\]
Thus, $w=s_2s_1$ is the required element.
Since $\tau$ is injective outside the zero section of $\phi$, it is still true that 
$\codim_{\gc(\lb)} \tau(\phi^{-1}(\co(w)))=2$.
This proves that  $\codim \fN_{U'}(\lb)\le 2$.

On the other hand, the  similar argument shows that if $w\in W^\lb$ and $\ell(w)=1$
(i.e., $w=s_i$, where $(\ap_i,\lb)\ne 0$), then $\dot w{\cdot}v_{-\lb}\not\in \fN_{U'}(\lb)$.
Therefore, $\codim \fN_{U'}(\lb)= 2$.
\end{proof}

\begin{cl}    \label{cl:eq-rk1}
Suppose that $U'\ne \{1\}$. Then
$\pi_{\lb,U'}: \gc(\lb) \to \gc(\lb)\md U'$ is equidimensional if and only if $\lb=a_i\varpi_i$
for some $i$. In particular, if the action of $G$ on $\gc(\lb)$ is effective and 
$\pi_{\lb,U'}$ is equidimensional, then $G$ is simple.
\end{cl}\begin{proof}
It follows from Theorem~\ref{thm:dim-HV}(i) that $\dim\gc(\lb)\md U'=2$
if and only if  $\lb=a_i\varpi_i$.
\end{proof}

Now, we turn to considering general monoids $\mathfrak S\subset \mathfrak X_+$.
For any $S\subset \mathfrak X$, let  $\mathsf{con}(S)$ denote the closed 
cone in $\mathfrak X_\BR$  generated by $S$.  

\begin{lm}    \label{lm:2-cones}
Suppose that we are given two monoids $\mathfrak S_1$ and $\mathfrak S_2$
such that $\mathsf{con}(\mathfrak S_1)=\mathsf{con}(\mathfrak S_2)$.
Then 
$\pi_{\mathfrak S_1,U'}$ is equidimensional if and only if $\pi_{\mathfrak S_2,U'}$ is.
\end{lm}
\begin{proof}
It suffices to treat the case in which $\mathfrak S_2=\mathsf{con}(\mathfrak S_1)\cap
\mathfrak X_+$. Then $\bbk[\gc(\mathfrak S_2)]$ is a finite 
$\bbk[\gc(\mathfrak S_1)]$-module \cite[Prop.\,4]{vp72}. 
Consider the commutative diagram
\[
\begin{array}{ccc} \gc(\mathfrak S_2) & \stackrel{\psi}{\longrightarrow} &
\gc(\mathfrak S_1) \\
\Big\downarrow\vcenter{%
\rlap{$\pi_{\mathfrak S_2,U'}$}} & & \Big\downarrow\vcenter{%
\rlap{$\pi_{\mathfrak S_1,U'}$}} \\
\gc(\mathfrak S_2)\md U'  &\stackrel{\psi\md U'}{\longrightarrow} & \gc(\mathfrak S_1)\md U' .\end{array}
\]
Here $\psi$ is finite, and it suffices to prove that $\psi\md U'$ is also finite, i.e.,
that $\bbk[\gc(\mathfrak S_2)]^{U'}$ is a finite 
$\bbk[\gc(\mathfrak S_1)]^{U'}$-module.
By the ``transfer principle'' 
(\cite[Ch.\,1]{br81}, \cite[\S\,3]{po86}),  we have
\[
   \bbk[X]^{U'}\simeq (\bbk[X]\otimes \bbk[G/U'])^G   
\]
for any affine $G$-variety $X$. Hence, one has to prove that
$(\bbk[\mathfrak S_2]\otimes \bbk[G/U'])^G$ is a finite 
$(\bbk[\mathfrak S_1]\otimes \bbk[G/U'])^G$-module, which readily follows from the fact that
$\bbk[G/U']$ is finitely generated and 
$G$ is reductive.
\end{proof}

\begin{thm}   \label{thm:main-EQ}
The quotient morphism $\pi_{\mathfrak S,U'}$ is equidimensional if and only if\/
$\mathsf{con}(\mathfrak S)$ is generated by a sparse set of fundamental weights.
\end{thm}
\begin{proof}
1) The ``if'' part readily follows from Lemma~\ref{lm:2-cones}  
and Theorem~\ref{thm:eq-sparse2}.

2) Suppose that $\pi_{\mathfrak S,U'}: \gc(\mathfrak S)\to \gc(\mathfrak S)\md U'$ is equidimensional.
By Lemma~\ref{lm:2-cones}, it  suffices to consider the case in which
$\mathfrak S=\mathsf{con}(\mathfrak S)\cap \mathfrak X_+$. 
Then $\gc(\mathfrak S)$ is normal (see Theorem~\ref{thm:vp}(3)).
Consider an arbitrary edge,
$\mathsf{con}(\lb)$, of $\mathsf{con}(\mathfrak S)$. It is assumed that
$\lb\in \mathfrak S$ is a primitive element of $\mathfrak X_+$.
By \cite[Prop.\,7]{vp72}, 
the \textsf{HV}-variety $\gc(\lb)$ is a subvariety of $\gc(\mathfrak S)$.
On the other hand, $\bbk[\gc(\lb)]=\bigoplus_{n\ge 0}\sfr(n\lb)$ is a $G$-stable
subalgebra
of $\bbk[\gc(\mathfrak S)]=\bigoplus_{\mu\in\mathfrak S}\sfr(\mu)$.
This yields the chain of $G$-equivariant maps
\[
       \gc(\lb)\hookrightarrow \gc(\mathfrak S) \stackrel{r}{\longrightarrow} \gc(\lb) .
\]
Here the composite map is the identity, i.e., $r$ is a $G$-equivariant retraction.  
Furthermore, passage to the 
subalgebras of $U'$-invariants (= quotient varieties) yields the maps
\[
      \gc(\lb)\md U'\hookrightarrow \gc(\mathfrak S)\md U' \stackrel{r\md U'}{\longrightarrow} \gc(\lb)\md U' ,
\]
which shows that $r\md U'$ is a retraction, too.  This also shows that both $r$ and $r\md U'$
are onto. Consider  the commutative diagram
\[
\xymatrix{
  \gc(\lb) \ar@{^{(}->}[r]  \ar[d]_{\pi_{\lb,U'}}    
 & \gc(\mathfrak S) \ar@{->>}[r]^r  \ar[d]_{\pi_{\mathfrak S,U'}}
 & \gc(\lb) \ar[d]_{\pi_{\lb,U'}} \\
 \gc(\lb)\md U' \ar@{^{(}->}[r]  & \gc(\mathfrak S)\md U' \ar@{->>}[r]^{r\md U'} & \gc(\lb)\md U'
}
\]
As $\gc(\mathfrak S)$ is normal, the same is true for $\gc(\mathfrak S)\md U'$. 
Since $\pi_{\mathfrak S,U'}$ is equidimensional and both $\gc(\mathfrak S)$ and
$\gc(\mathfrak S)\md U'$ are conical, it follows from Lemma~\ref{klod} that 
$\pi_{\mathfrak S,U'}$
is onto. Therefore, $\pi_{\lb,U'}$
is  onto as well. Furthermore,
$\pi_{\lb,U'}=\pi_{\mathfrak S,U'}\vert_{\gc(\lb)}$, since $\gc(\lb)$ is  a $G$-stable 
subvariety of $\gc(\mathfrak S)$.  This shows that $\pi_{\mathfrak S,U'}(\gc(\lb))$ is 
a closed subset of $ \gc(\mathfrak S)\md U'$.

Let $Y\subset \gc(\mathfrak S)$ be an irreducible component of 
$\pi_{\mathfrak S,U'}^{-1}(\pi_{\mathfrak S,U'}(\gc(\lb)))$ that contains $\gc(\lb)$ and
maps dominantly to $\pi_{\mathfrak S,U'}(\gc(\lb))$.
Consider  the commutative diagram
\[
\xymatrix{
Y\ar[rr]^{r\vert_Y}  \ar[dr]_{\pi_{\mathfrak S,U'}\vert_Y  } && \gc(\lb) \ar[dl]^{\pi_{\mathfrak S,U'}\vert_{\gc(\lb)}} \\
&  \pi_{\mathfrak S,U'}(\gc(\lb)) &   
}
\]
By the very construction of $Y$, the morphism $r\vert_Y$ is onto and  $\pi_{\mathfrak S,U'}\vert_Y$ is equidimensional. It follows that $\pi_{\mathfrak S,U'}\vert_{\gc(\lb)}$ is also
equidimensional. 
Consequently, $\pi_{\lb,U'}=\pi_{\mathfrak S,U'}\vert_{\gc(\lb)}$ is equidimensional and, by Corollary~\ref{cl:eq-rk1}, $\lb=\varpi_i$ for some $i$ (recall that $\lb$ is supposed to be primitive).
Thus, the edges of $\mathsf{con}(\mathfrak S)$ are generated by fundamental weights.
Finally, by Theorem~\ref{thm:eq-sparse1}, the corresponding set of fundamental weights is sparse. 
\end{proof}

\begin{rmk}
Our proof of the ``only if'' part exploits ideas of Vinberg and Wehlau for the equidimensional 
quotients by $G$ (see \cite[Theorem\,8.2]{t55} and \cite[Prop.\,2.6]{we93}).
\end{rmk}

\begin{rmk}
We can prove a general equidimensionality criterion for the quotients of $\eus S$-varieties 
by $U$. This topic will be considered in a forthcoming publication.
\end{rmk}

\section{The Hilbert--Mumford criterion for $U'$}    \label{sect:HM}

\noindent
Let $X$ be an irreducible affine $G$-variety  and $x_0\in X^G$.
For any $H\subset G$, define the {\it null-cone\/} with respect to $H$ and $x_0$ as 
\[
      \fN_{H}(X)=\{x\in X \mid  F(x)=F(x_0) \quad \forall F\in\bbk[X]^{H}\} .
\]
If $\bbk[X]^H$ is finitely generated, then $\fN_H(X)$ can be regarded as 
the fibre of $\pi_{X,H}$ containing $x_0$.
Below, we give a characterisation of $\fN_{U'}(X)$ via one-parameter subgroups
({\sf 1-PS} for short) of $T$. This is inspired by
Brion's description of null-cones for $U$-invariants  \cite[Sect.\,IV]{br83}.
Recall that the Hilbert--Mumford criterion for $G$ asserts that 

\centerline{
{\it $x\in \fN_G(X)$ if and only if there is a 1-PS $\tau: \bbk^\times \to G$
such that\/ $\lim_{t\to 0}\tau(t){\cdot}x=x_0$}  }
\noindent (cf. \cite[III.2]{kr84}, \cite[\S\,5.3]{t55}).
By \cite[Theorem\,2.2]{odno-sv}, there is the canonical affine model of the homogeneous 
space $G/U'$, that is, an affine pointed $G$-variety $(\ov{G/U'}, \bfp)$  such that
\begin{itemize}
\item   \ $G_\bfp=U'$;
\item  \  $G{\cdot}\bfp$ is dense in  $\ov{G/U'}$;
\item  \  $\bbk[\ov{G/U'}]=\bbk[G]^{U'}$.
\end{itemize}
Here $\bfp=(f_1,\tilde f_1,\dots, f_r,\tilde f_r)$ is a direct sum of weight vectors in 
$2\sfr(\varpi_1)\oplus \cdots \oplus 2\sfr(\varpi_r)$, 
with weights $\varpi_i, \varpi_i-\ap_i$ ($1\le i\le r$).
If $G$ has no simple factors $SL_2, SL_3$, then all these weights belong to an open 
half-space of $\mathfrak X_\BR$ (see the proof of \cite[Prop.\,1.9]{odno-sv}). In this case, 
$\bfp$ is unstable  and $\ov{G/U'}$ contains the origin in  
$2\sfr(\varpi_1)\oplus \cdots \oplus 2\sfr(\varpi_r)$.
Let $\tau: \bbk^\times \to T$ be a {\sf 1-PS}. Using the canonical pairing between 
$\mathfrak X$ and the set of
{\sf 1-PS} of $T$, we will regard $\tau$ as an element of $\mathfrak X_\BR$.
Let us say that $\tau$ is 
$U'$-{\it admissible\/}, if $(\tau, \varpi_i)>0$ and $(\tau, \varpi_i-\ap_i)>0$ for all $i$;
that is, if \ $\lim_{t\to 0}\tau(t){\cdot}\bfp=0$. Since $\bbk[\ov{G/U'}]=\bbk[G]^{U'}$,
one has the  isomorphism  
\beq   \label{eq:isom-U'}
\bbk[X\times \ov{G/U'}]^G= (\bbk[X]\otimes \bbk[G]^{U'})^G\isom \bbk[X]^{U'}
\eeq
that  takes $\tilde F(\cdot ,\cdot )\in \bbk[X\times \ov{G/U'}]^G$ to 
$F(\cdot)=\tilde F(\cdot ,\bfp)\in \bbk[X]^{U'}$.

\begin{thm}   \label{thm:1-ps}
Suppose that $G$ has no simple factors $SL_2, SL_3$. Then the following conditions are equivalent:
\begin{itemize}
\item[\sf (i)]  \   $x\in \fN_{U'}(X)$, i.e., $F(x)=F(x_0)$ for all $F\in\bbk[X]^{U'}$;
\item[\sf (ii)] \  there is $u\in U$ and a $U'$-admissible 1-PS 
$\tau: \bbk^\times \to T$ such that $\lim_{t\to 0}  \tau(t)u{\cdot}x=x_0$.
\end{itemize}
\end{thm}\begin{proof}
$\mathsf{(i)} \Rightarrow \mathsf{(ii)}$.
Suppose that $x\in \fN_{U'}(X)$. Then 
$\tilde F(x ,\bfp)=F(x)=F(x_0)=\tilde F(x_0 ,\bfp)$. Since $\bfp$ is unstable in $\ov{G/U'}$,
we have $\tilde F(x_0 ,\bfp)=\tilde F(x_0 ,0)$. Thus, $\tilde F(x ,\bfp)=\tilde F(x_0 ,0)$ for
all $\tilde F\in (\bbk[X]\otimes \bbk[G]^{U'})^G$, i.e.,
$(x,\bfp)\in \fN_G(X\times \ov{G/U'})$.
By the Hilbert--Mumford criterion for $G$, there is a {\sf 1-PS}  
$\nu:\bbk^\times \to G$ such that  $\nu(t){\cdot}(x,\bfp)\underset{t\to 0}{\longrightarrow}(x_0,0)$.

By a result of Grosshans \cite[Cor.\,1]{gro82} (see also \cite[IV.1]{br83}), 
we may assume that $\nu(\bbk^\times)\subset B$.
Then there is $u\in U$ such that  $\tau(t):=u\nu(t)u^{-1} \in T$.
Therefore,
\[
     \tau(t)u{\cdot}(x,\bfp)  \underset{t\to 0}{\longrightarrow}(x_0,0) .
\] 
Note that $u{\cdot}\bfp$ ($u\in U$) does not differ much from $\bfp$.
Namely, each component $f_i$ remains intact, whereas $\tilde f_i$ is replaced with
$\tilde f_i + c_if_i$ for some $c_i\in \bbk$.  This means that 
$\tau(t)u{\cdot}\bfp\underset{t\to 0}{\longrightarrow} 0$ if and only if
$\tau(t){\cdot}\bfp\underset{t\to 0}{\longrightarrow} 0$. That is, $\tau$ is actually $U'$-admissible
and $\lim_{t\to 0}  \tau(t)u{\cdot}x=x_0$.

\noindent 
$\mathsf{(ii)} \Rightarrow \mathsf{(i)}$. \ 
Suppose that $F\in \bbk[X]^{U'}$ and $\tilde F$ is the corresponding $G$-invariant in 
$\bbk[X\times \ov{G/U'}]$. Then   
$F(x)=\tilde F(x,\bfp)=\tilde F(\tau(t)u{\cdot}x,\tau(t)u{\cdot}\bfp)$.  
Since $u{\cdot}\bfp$ is  a linear 
combination of weight vectors with the same weights and $\tau$ is $U'$-admissible, we 
have  $\lim_{t\to 0}\tau(t)u{\cdot}\bfp=0$.  Hence 
$F(x)=\tilde F(x_0, 0)=\tilde F(x_0,\bfp)=F(x_0)$.
\end{proof}

\begin{rmk}
Our Theorem~\ref{thm:1-ps} is  similar to Theorem~5 in \cite{br83} on null-cones
for $U$-invariants. The only difference is that we end up with a smaller class of 
admissible {\sf 1-PS}.
\end{rmk}

Obviously, there are inclusions $ \fN_{U'}(X)\subset \fN_U(X)\subset \fN_G(X)$ and hence
\[
   G{\cdot}\fN_{U'}(X)\subset G{\cdot}\fN_U(X)\subset \fN_G(X) .
\]
It is proved in \cite[Th\'eor\`eme\,6(ii)]{br83} that actually $G{\cdot}\fN_U(X)=\fN_G(X)$. 
Below, we investigate the similar problem  for $U'$.

Recall that $\mathsf{con}(S)$ is the closed 
cone in $\mathfrak X_\BR$  generated by $S$.  If $K\subset \mathfrak X_\BR$ is a closed cone, then
$K^\perp$ denotes the dual cone and $K^o$ denotes the relative interior of $K$.
By the very definition, the cone generated by the $U'$-admissible {\sf 1-PS} is open, and
its closure is dual to $\mathsf{con}(\{\varpi_i, \varpi_i-\ap_i \mid  i=1,\dots, r\})$.
By \cite[Theorem\,4.2]{odno-sv}, we have 
\[
   \mathsf{con}(\{\varpi_i, \varpi_i-\ap_i \mid  i=1,\dots, r\})^\perp={\mathsf{con}}(\Delta^+\setminus\Pi).
\]
Hence the cone generated by the $U'$-admissible {\sf 1-PS} equals
$\mathsf{con}(\Delta^+\setminus\Pi)^o$.

\begin{thm}  \label{thm:sovpad-null}
Suppose that $G$ has no simple factors of type $SL$. Then 

1)  \ $\mathsf{con}(\varpi_1,\dots,\varpi_r)\subset \mathsf{con}(\Delta^+\setminus\Pi)$,

2) \ $G{\cdot}\fN_{U'}(X)=\fN_G(X)$ for all affine $G$-varieties $X$.
\end{thm}
\begin{proof}
1)  Taking the dual cones yields the equivalent condition that
\[
    \mathsf{con}(\{\varpi_i, \varpi_i-\ap_i \mid  i=1,\dots, r\})\subset \mathsf{con}(\Delta^+) .
\]
That is, one has to verify that each $\varpi_i-\ap_i$ has non-negative coefficients 
in the expression via the simple roots. 
Let $C$ denote the Cartan matrix of a simple group $G$.  
All the entries of $C^{-1}$ are positive and the rows of $C^{-1}$  provide the expressions
of the fundamental weights via the simple roots.  Hence it remains to check that 
the diagonal entries of 
$C^{-1}$ are $\ge 1$. An explicit verification shows that this is true if
$G\ne SL_{r+1}$.
(The matrices $C^{-1}$ can be found  in \cite[Table\,2]{vo}.)
 
2)  Suppose that $x\in\fN_G(X)$. Then there exist $g\in G$ and $\tau: \bbk^\times \to T$ such that
$\lim_{t\to 0} \tau(t)g{\cdot}x=x_0$. Let $y=g{\cdot}x$.  The set of all {\sf 1-PS} \
$\nu:\bbk^\times \to T$ such that $\lim_{t\to 0} \nu(t){\cdot}y=x_0$ generates an open 
cone in $\mathfrak X_\BR$.  Therefore, we may assume that $\tau$ is a regular 
{1-PS}. Now,
in view of the Hilbert--Mumford criterion for $G$ and Theorem~\ref{thm:1-ps}, 
it suffices to prove that any regular {\sf 1-PS} 
of $T$ is $W$-conjugate
to a $U'$-admissible one. This follows from part 1), since 
$\mathsf{con}(\varpi_1,\dots,\varpi_r)$ is a fundamental domain for the $W$-action on
$\mathfrak X_\BR$ and 
$\mathsf{con}(\varpi_1,\dots,\varpi_r)^o\subset \mathsf{con}(\Delta^+\setminus\Pi)^o$.
\end{proof}

\noindent
For $G=SL_{r+1}$, we have $\varpi_1-\ap_1, \varpi_r-\ap_r \not\in \mathsf{con}(\Delta^+)$
and therefore, 
$\mathsf{con}(\varpi_1,\dots,\varpi_r)\not\subset \mathsf{con}(\Delta^+\setminus\Pi)$.
More precisely, \ $\varpi_1,\varpi_r \not\subset \mathsf{con}(\Delta^+\setminus\Pi)$.
This means that one may expect that, for {\sl some\/} $SL_{r+1}$-varieties, there is the 
strict inclusion $G{\cdot}\fN_{U'}(X)\subsetneqq \fN_G(X)$.

\begin{ex}
For $m\ge 3$, consider the representation of $G=SL_3$ in the space $V=\sfr(m\varpi_1)$ 
of  forms of  degree $m$ in three variables $x,y,z$. By Theorem~\ref{thm1.6}, 
$\dim V^{U'}=m+1$. Let $U$ be the subgroup of the unipotent 
upper-triangular  matrices in the basis  dual to $(x,y,z)$. 
The $U'$-invariants of degree 1 are the coefficients of $x^m, x^{m-1}y,\dots, xy^{m-1}, y^m$.
Therefore,  $\fN_{U'}(V)$ is contained in the subspace
of forms having the linear factor $z$ and all the forms in 
$SL_3{\cdot}\fN_{U'}(V)$ have a linear factor. On the other hand, the null-form (with respect 
to $SL_3$) $x^m+y^{m-1}z$ is irreducible. Hence, 
$SL_3{\cdot}\fN_{U'}(V)\ne \fN_{SL_3}(V)$.
\end{ex}

\begin{rema}
In view of Theorem~\ref{thm:1-ps}, it would be much more instructive to have 
such an example for $SL_n$, $n\ge 4$. However, we are unable to provide it yet.
\end{rema}

\section{Equidimensional quotients and irreducible representations of simple groups}    
\label{sect:classif}

\noindent
In this section, we  transform the criterion of Theorem~\ref{thm:main-eq-S} in a
sufficient condition applicable to a wider class of $G$-varieties. Then  we obtain the list of irreducible representations $V$ of simple algebraic groups $G\ne SL_2$ such that  $\bbk[V]$ is a free $\bbk[V]^{U'}$-module.

For any affine irreducible 
$G$-variety $Z$, there is a flat degeneration  $\bbk[Z] \leadsto \gr(\bbk[Z])$.
(Brion attributes this to  Domingo~Luna in his thesis, see \cite[Lemma\,1.5]{br81}). 
Here $\gr(\bbk[Z])$
is  again a finitely generated $\bbk$-algebra and a locally-finite $G$-module, and  
$\gr Z:=\spe(\gr(\bbk[Z]))$ is an affine horospherical 
$G$-variety. The whole theory of ``contractions of actions of reductive groups'' 
is later developed in \cite{po86}. (See also 
\cite{br86}, \cite{vin86}, \cite{pa97} for related results and other applications.)
The ``contraction'' $Z\leadsto \gr Z$ has the property that the algebras $\bbk[Z]$ and
$\bbk[\gr Z]=\gr(\bbk[Z])$ are isomorphic as  $G$-modules.
But the multiplication in $\bbk[\gr Z]$ is simpler than that in $\bbk[Z]$;
namely, if $M$ and $N$ are two simple $G$-modules in $\bbk[\gr Z]$, then
$M{\cdot}N$ (the product in $\bbk[\gr Z]$) is again a simple $G$-module.
Furthermore, $\bbk[\gr Z]^U\simeq \bbk[Z]^U$ and $G{\cdot}((\gr Z)^U)=\gr Z$.
This means that if $Z$ is a spherical $G$-variety, then $\gr Z$ is an $\eus S$-variety.

\begin{thm}    \label{thm:deformation}
Suppose that $G$ acts on  an irreducible affine variety $X$ such that 
(1) \ $\bbk[X]^{U}$ is a polynomial
algebra and  (2) \ the weights of free generators are fundamental, different and form a sparse set.
Then $\bbk[X]^{U'}$ is also polynomial, of Krull dimension $2\dim X\md U$,
and the quotient $\pi_{X,U'}: X\to X\md U'$ is equidimensional.   
\end{thm}
\begin{proof}
The idea is the same as in the proof of the similar result for $U$-invariants in 
\cite[Theorem\,5.5]{aura}. We  use the fact that in our situation
$\gr X$ is an $\eus S$-variety whose monoid of dominant weights is 
generated by a sparse set of fundamental weights.

Let $\varpi_{1},\dots, \varpi_{m}$ be the weights of free generators of $\bbk[X]^U$.
Set  $\Gamma=\sum_{i=1}^m \BN \varpi_{i}$.
It follows from the hypotheses on weights that $\bbk[X]$ is a multiplicity free 
$G$-module, i.e., $X$ is a spherical $G$-variety \cite[Theorem\,2]{VK}. Therefore, 
$\bbk[X]$ is isomorphic to $\bigoplus_{\lb\in\Gamma}\sfr(\lb)$ as $G$-module
and $\gr X\simeq \gc(\Gamma)$.

By \cite[\S 5]{po86}, there exists a $G$-variety $Y$ and a  function $q\in\bbk[Y]^{G}$
such that $\bbk[Y]/(q-a)\simeq \bbk[X]$ for all $a\in \bbk^\times$, $\bbk[Y][q^{-1}]
\simeq \bbk[X][q,q^{-1}]$, and $\bbk[Y]/(q) \simeq \bbk[\gr X]$.
Recall some details on constructing $Y$ and $\gr X$.
Let $\varrho$ be the half-sum of the positive coroots. 
For $\lb\in \mathfrak X_+$, we set $\hot(\lb)=(\lb, \varrho)$. 
Letting  $\bbk[X]_{(n)}=\bigoplus_{\lb:\, \hot(\lb)\le n} \sfr(\lb)$, one obtains
an ascending filtration of the algebra $\bbk[X]$:
\[ 
\{0\}\subset \bbk[X]_{(0)}\subset \bbk[X]_{(1)}\subset \cdots \subset \bbk[X]_{(n)}\cdots .
\]
Each subspace $\bbk[X]_{(n)}$ is $G$-stable and finite-dimensional and
$\bbk[X]_{(0)}=\bbk[X]^G=\bbk$. Let $q$ be a formal variable.
Then the algebras 
$\bbk[Y]$ and $\gr(\bbk[X])$ are defined as follows:
\begin{gather*}
\bbk[Y]= \bigoplus_{n=0}^\infty \bbk[X]_{(n)}q^n \subset \bbk[X][q] \ ,
\\
    \gr(\bbk[X])= \bigoplus_{n\ge 0} \bbk[X]_{(n)}/\bbk[X]_{(n-1)} \ .
\end{gather*}
Let $f_1,\ldots,f_m$ be the free generators of $\bbk[X]^{U}$, where $f_i\in \sfr(\varpi_i)^U$, as usual.
They can also be regarded as free generators of $\bbk[\gr X]^U$.
By Theorem~\ref{thm:some-fund}, $\bbk[\gr X]^{U'}$ is freely generated by 
$f_1,\tilde f_1,\dots,f_m,\tilde f_m$ and 
by Theorem~\ref{thm:main-eq-S}, 
$\pi_{\gr X,U'} : \gr X \to (\gr X)\md U'$ is equidimensional.
On the other hand, 
it follows from  \cite[Theorem\,2.4]{odno-sv} that $f_1,\tilde f_1,\dots,f_m,\tilde f_m$
also generate $\bbk[X]^{U'}$. 
Therefore, to conclude that $\bbk[X]^{U'}$ is polynomial, it suffices to know that
$\dim X\md U'=\dim (\gr X)\md U' (=2m)$. To this end, we exploit the following facts:

\noindent
a) \ For an irreducible  $G$-variety $X$, 
there always exists a generic stabiliser for the $U$-action
on $X$ \cite[Corollaire\,1.6]{BLV}, which we denote by $\mathsf{g.s.}(U{:}X)$;

\noindent
b) \ If $X$ is affine, then this generic stabiliser depends only on the $G$-module structure of 
$\bbk[X]$, i.e., on the highest weights of $G$-modules occurring in 
$\bbk[X]$ \cite[Theorem\,1.2.9]{disser}. Consequently,
$\mathsf{g.s.}(U{:}X)=\mathsf{g.s.}(U{:}\gr X)$;

\noindent
c) \  the minimal dimension of $U'$-stablisers in $X$ equals 
$\dim (U' \cap \mathsf{g.s.}(U{:}X))$ \cite[Lemma\,2.5]{odno-sv}. Therefore it is the same
for $X$ and $\gr X$;

\noindent
d) \  Since $U'$ is unipotent,  we have
$\dim X\md U'=\dim X-\dim U'+ \min_{x\in X}\dim U'_x$.

Combining a)-d) yields the desired equality and thereby the assertion that 
$\bbk[X]^{U'}$ is polynomial, of Krull dimension $2m=2\dim X\md U$.

Let $n_i$ 
be the smallest integer such that $\sfr(\varpi_i)\subset \bbk[X]_{(n_i)}$. Using the above 
description of $\bbk[Y]$ and $\bbk[\gr X]^{U'}$, one easily obtains that 
\begin{gather*}
  \bbk[Y]^{U}= \bbk[q,q^{n_1}f_1,\,\ldots,q^{n_m}f_m] \\
  \bbk[Y]^{U'}= \bbk[q,q^{n_1}f_1,\,q^{n_1}\tilde f_1,\ldots,q^{n_m}f_m, \,q^{n_m}\tilde f_m],
\end{gather*}
i.e., both algebras are
polynomial, of Krull dimension $m+1$ and $2m+1$, respectively. 
By a result of Kraft, the first equality implies that $Y$ has rational singularities (see 
\cite[Theorem\,1.6]{br81}, \cite[Theorem\,6]{po86}).
One has the following commutative diagram:
\[ \begin{array}{ccccccc}
C(\Gamma) & \simeq & \gr X & \hookrightarrow & Y & \leftarrow & 
X\times{\mathbb A}^1 \\
  & & \Big\downarrow\vcenter{%
\rlap{$\pi_{\gr X,U'}$}} & & \Big\downarrow\vcenter{%
\rlap{$\pi_{Y,U'}$}} & &      \\
{\mathbb A}^{2m} & \simeq & (\gr X)\md U' & \hookrightarrow & Y\md U' & \simeq &
{\mathbb A}^{2m+1} \\
  & & \Big\downarrow & & \Big\downarrow\vcenter{%
\rlap{$q$}} & &    \\
  & & \{0\} & \hookrightarrow & {\mathbb A}^1 & & 
\end{array}
\]
Consequently,
\[
  \fN_{U'}(\gr X)=\pi_{\gr X,U'}^{-1}(\pi_{\gr X,U'}(\bar 0))=\pi_{Y,U'}^{-1}(\pi_{Y,U'}(\bar 0))=\fN_{U'}(Y) ,
\]
where $\bar 0\in \gr X \subset Y$ is the unique $G$-fixed point of $\gr X$. 
Since $\dim Y=\dim X+1$, $\dim Y\md U'=\dim (\gr X)\md U' +1$, 
and $\pi_{\gr X,U'}$ is equidimensional,  the morphism $\pi_{Y,U'}$ is equidimensional 
as well.
As $Y$ has rational singularities and hence is Cohen-Macaulay, this implies that
 $\bbk[Y]$ is a flat  $\bbk[Y]^{U'}$-module. Since 
$\bbk[Y][q^{-1}]\simeq\bbk[X][q,q^{-1}]$ and 
$\bbk[Y]^{U'}[q^{-1}]\simeq\bbk[X]^{U'}[q,q^{-1}]$,
we conclude that $\bbk[X]$ is a flat $\bbk[X]^{U'}$-module. Thus, $\pi_{X,U'}$ is equidimensional. 
\end{proof}

Our next goal is to obtain the list of all irreducible representations $V$ 
of simple algebraic groups such that $\bbk[V]$ is a free $\bbk[V]^{U'}$-module.
As is well known, $\bbk[V]$ is a free $\bbk[V]^{U'}$-module if and only if
$\bbk[V]^{U'}$ is polynomial and $\pi_{V,U'}$ is equidimensional \cite[Prop.\,17.29]{gerry}.
Therefore, the required
representations are contained  in \cite[Table\,1]{odno-sv} and
our task is to pick from that table the representations having the additional property that 
$\pi_{V,U'}$ is equidimensional.  The numbering of fundamental weights of simple algebraic
groups follows \cite[Tables]{vo}.

\begin{thm}   \label{thm:classif-eq}
Let $G$ be a connected simple algebraic  group with $\rk G\ge 2$ and 
$\sfr(\lb)$ a simple $G$-module. The following conditions are equivalent:
\begin{itemize}
\item[\sf (i)] \  $\bbk[\sfr(\lb)]$ is a free $\bbk[\sfr(\lb)]^{U'}$-module;
\item[\sf (ii)] \  Up to symmetries of the Dynkin diagram of $G$, the pairs $(G,\lb)$  occur
in the following list:
$(\GR{A}{r}, \varpi_1)$,\,$(\GR{B}{r}, \varpi_1)$,\,$(\GR{C}{r}, \varpi_1)$, $r\ge 2$; \\
$(\GR{D}{r}, \varpi_1)$, $r\ge 3$;
\\$(\GR{B}{3}, \varpi_3)$,\,$(\GR{B}{4}, \varpi_4)$,\,$(\GR{D}{5}, \varpi_5)$,\,$(\GR{E}{6}, \varpi_1)$,\,$(\GR{G}{2}, \varpi_1)$.
\end{itemize}
\end{thm}\begin{proof}
{\sf (ii)}$\Rightarrow${\sf (i)}. By \cite[Theorem\,5.1]{odno-sv}, all these representations 
have a polynomial algebra of $U'$-invariants.
Consider $X=\fN_G(\sfr(\lb))$,  the 
null-cone with respect to $G$. The nonzero weights of generators of
$\bbk[\sfr(\lb)]^U$ (and hence the weights of generators of $\bbk[X]^U$) 
given by Brion \cite[p.\,13]{br83} are fundamental and form a sparse set. Consequently,  
Theorem~\ref{thm:deformation} applies to $X$, and  $\pi_{X,U'}$ is equidimensional.
Since $X$ is either a $G$-invariant hypersurface in $\sfr(\lb)$ or equal to $\sfr(\lb)$,
$\pi_{\sfr(\lb),U'}$ is also equidimensional.

{\sf (i)}$\Rightarrow${\sf (ii)}. We have to prove that, for the other items in 
\cite[Table\,1]{odno-sv}, the quotient is not equidimensional. 
The list of such ``bad'' pairs $(G,\lb)$  is:
$(\GR{A}{r},\varpi_2^*)$ with $r\ge 4$;
$(\GR{B}{5}, \varpi_5)$,\,
$(\GR{D}{6}, \varpi_6)$,\,
$(\GR{E}{7}, \varpi_1)$,\,
$(\GR{F}{4}, \varpi_1)$.   Note that $(\GR{A}{3},\varpi_2^*)=(\GR{D}{3}, \varpi_1)$ and this good pair is included in the list in (ii). 
\\
It suffices to 
check that the free generators of $\bbk[\sfr(\lb)]^{U'}$ given in that Table
do not form a regular sequence. 
To this end, we point out a certain relation
in $\bbk[\sfr(\lb)]$
using the fact the weights of generators do not form a sparse set
(cf.  the proof of Theorem~\ref{thm:eq-sparse1}).

The only ``bad'' serial case is  $(\GR{A}{r},\varpi_2^*)$ with $r\ge 4$.
The algebra $\bbk[\sfr(\varpi^*_2)]^{U}$ has free generators 
$f_{2i} $ ($1\le i\le [r/2]$) of  degree $i$ and weight $\varpi_{2i}$,
and for $r$ odd,  there is also the Pfaffian, which is $G$-invariant.
Then $\bbk[\sfr(\varpi^*_2)]^{U'}$ is freely generated by 
$f_2,\tilde f_2,f_4,\tilde f_4,\dots$ (and the Pfaffian, if $r$ is odd).
Using the 4-nodes fragments  of the weight posets $\eus P(\varpi_2)$
and $\eus P(\varpi_4)$
and notation of the proof of Theorem~\ref{thm:eq-sparse1}, we construct a 
$U$-invariant function $f_2q_4-\tilde f_2p_4+p_2\tilde f_4-q_2f_4$
of degree $3$ and weight $\varpi_2+\varpi_4-\ap_2-\ap_3-\ap_4=\varpi_1+\varpi_5$.
(Cf. Eq.~\eqref{eq:sootnoshenie}.)
However, there are no such nonzero $U$-invariants in $\bbk[\sfr(\varpi_2^*)]$.
This yields a relation in $\bbk[\sfr(\varpi^*_2)]$ involving free generators 
$f_2,\tilde f_2,f_4,\tilde f_4 \in \bbk[\sfr(\varpi^*_2)]^{U'}$.

In all other cases, we can do the same thing using a pair of generators of $\bbk[\sfr(\lb)]^{U}$
corresponding to suitable fundamental  weights.  The only difference is that one of 
these two $U$-invariants is \un{not} included in the minimal generating system of 
$\bbk[\sfr(\lb)]^{U'}$ and should be expressed via some other $U'$-invariants. 
Nevertheless, the resulting relation still shows that the $U'$-invariants involved do not 
form a regular sequence.

For instance, consider the pair $(\GR{D}{6},\varpi_6)$. Here the free generators of 
$\bbk[\sfr(\varpi_6)]^{U}$ have the following degrees and weights:
$(1,\varpi_6),\,(2,\varpi_2),\,(3,\varpi_6),\,(4,\varpi_4),\,(4,\un{0})$ \cite{br83}. The invariants
themselves are denoted by $f^{(1)}_6,\, f_2,\, f^{(3)}_6,\, f_4,\, F$, respectively. 
Starting with the $U$-invariants
$f_2$ and $f_4$, we obtain, as a above, a relation of the form 
\beq    \label{eq:f2-f4}
     f_2q_4-\tilde f_2p_4+p_2\tilde f_4-q_2f_4=0
\eeq  
in $\bbk[\sfr(\varpi^*_6)]$. However, $f_4$
is not a generator in $\bbk[\sfr(\varpi_6)]^{U'}$. Taking the second $U'$-invariant in each fundamental $G$-submodule, we obtain nine functions 
$f^{(1)}_6,\, \tilde f^{(1)}_6$, $f_2,\, \tilde f_2$, $f^{(3)}_6,\, \tilde f^{(3)}_6$, $f_4,\, \tilde f_4,\, F$
that generate $\bbk[\sfr(\varpi_6)]^{U'}$.  Here 
$f_4=f^{(1)}_6 \tilde f^{(3)}_6-\tilde f^{(1)}_6 f^{(3)}_6$ and the remaining eight functions freely
generate $\bbk[\sfr(\varpi_6)]^{U'}$.  Substituting this expression for $f_4$ in 
\eqref{eq:f2-f4}, we finally obtain the relation
\[
   f_2q_4-\tilde f_2p_4+p_2\tilde f_4-q_2\bigl(f^{(1)}_6 \tilde f^{(3)}_6-\tilde f^{(1)}_6 f^{(3)}_6\bigr)=0 ,
\]
which shows that the free generators of $\bbk[\sfr(\varpi_6)]^{U'}$ do not form a regular sequence.
\end{proof}

\un{Some} \un{o}p\un{en} p\un{roblems}.  Let $V$ be a rational $G$-module.

1$^o$. {\it Suppose that $V\md U$ is an affine space. Is it true that $V\md U'$ is a complete
intersection?} 

2$^o$. {\it Suppose that $V\md U'$ is an affine space and $G$ has no simple factors 
$SL_2$. Is it true that $V\md U$ is an affine space?}  
(In \cite{odno-sv}, we have proved that $V\md G$
is an affine space, but this seems to be too modest.)

\noindent
Direct computations provide an affirmative answer to both questions if $G$ is simple and
$V$ is a simple $G$-module.

\noindent
{\small  {\bf Acknowledgements.} Part of this work was done while I was visiting
MPIM (Bonn). I thank the Institute for the hospitality and inspiring environment.
I am grateful to E.B.~Vinberg for sending me the preprint \cite{vg08}.}

\end{document}